\documentclass[a4paper,12pt,twoside]{article}
\usepackage{verbatim} 
\usepackage{amsbsy}
\usepackage{amssymb}
\usepackage{amsfonts}
\usepackage{amsmath}
\usepackage{amsopn}
\usepackage{amsthm}
\usepackage[english]{babel}
\usepackage[T1]{fontenc}

\newtheorem{corollary}{Corollary}

\newtheorem{lemma}{Lemma}

\newtheorem{proposition}{Proposition}
\newtheorem{remark}{Remark}

\newtheorem{teo}{Theorem}

 \DeclareMathOperator{\er}{\mathbf{E}}
\DeclareMathOperator{\pr}{\mathbf{P}}

\DeclareMathOperator{\re}{\mathbb{R}}

\DeclareMathOperator{\p}{\mathbb{P}}
\DeclareMathOperator{\e}{\mathbb{E}}

\newcommand{\R}{\mbox{\rm I\hspace{-0.02in}R}}

\newcommand{\pox}{\mbox{\raisebox{-1.1ex}{$\rightarrow$}\hspace{-.15in}$X$}}

\newcommand{\finfg}{\mbox{\raisebox{-1.1ex}{$\rightarrow$}\hspace{-.15in}$G$}}
\newcommand{\finfd}{\mbox{\raisebox{-1.1ex}{$\rightarrow$}\hspace{-.15in}$D$}}

\newcommand{\ed}{\stackrel{(d)}{=}}

\numberwithin{equation}{section}
\title{
\textbf{Exact and asymptotic $n$-tuple laws at first and last passage}
}
\author{\textbf{A. E. Kyprianou\footnote{Department of Mathematical Sciences, University of Bath, Claverton Down, Bath, BA2 7AY. Email: a.kyprianou@bath.ac.uk, \, jcpm20@bath.ac.uk},} \footnote{Corresponding author.}\, \textbf{J. C. Pardo$^{*}$} and \textbf{ V. Rivero}\footnote{Centro de Investigaci\'on en Matem\'aticas A.C. Calle Jalisco s/n. 36240 Guanajuato, M\'exico. Email: rivero@cimat.mx}}
\date{\footnotesize This version: \today}

\begin{document}
\maketitle

\begin{abstract} Understanding the space-time features of how a L\'evy process crosses a constant barrier for the first time, and indeed the last time, is a problem which is central to many models in applied probability such as queueing theory, financial and actuarial mathematics, optimal stopping problems, the theory of branching processes to name but a few.
In \cite{KD} a new quintuple law was established for a general L\'evy process at first passage below a fixed level. In this article we use the quintuple law to establish a family of related joint  laws, which we call $n$-tuple laws, for L\'evy processes, L\'evy processes conditioned to stay positive and positive self-similar Markov processes at both first and last passage over a fixed level. Here the integer $n$ typically ranges from three to seven. Moreover, we look at asymptotic overshoot and undershoot distributions and relate them to overshoot and undershoot distributions of positive self-similar Markov processes issued from the origin. Although the relation between the $n$-tuple laws for L\'evy processes and positive self-similar Markov processes are straightforward thanks to the Lamperti transformation, by inter-playing the role of a (conditioned) stable processes as both a (conditioned) L\'evy processes and a positive self-similar Markov processes, we obtain a suite of completely explicit first and last passage identities for so-called Lamperti-stable L\'evy processes. This leads further to the introduction of a more general family of L\'evy processes which we call hypergeometric L\'evy processes, for which similar explicit identities may be considered.

\end{abstract}

\bigskip


\section{Introduction}
This paper concerns the joint laws of overshoots and undershoots of L\'{e}vy processes at
first and last upwards passage times of a constant boundary  leading to new general and explcit identities. We will therefore begin by
introducing some necessary but standard notation.

In the sequel $X=\{X_t : t\geq 0\}$ will always denote a L\'{e}vy process defined on the
filtered space $(\Omega ,\mathcal{F},\mathbb{F},\mathbb{P})$ where the filtration $%
\mathbb{F}=\{\mathcal{F}_{t}:t\geq 0\}$ is assumed to satisfy the usual
assumptions of right continuity and completion. Its characteristic exponent
will be given by $\Psi (\theta ):=-\log \mathbb{E}(e^{i\theta X_{1}})$ and its jump
measure by $\Pi _{X}$. Associated to the L\'evy measure $\Pi_{X}$ we define the left and right tail, $\overline{\Pi}^{-}_{X}$ and $\overline{\Pi}^{+}_{X},$ respectively, as follows
$$\overline{\Pi}^{-}_{X}(x)=\Pi_{X}(-\infty,-x),\qquad \overline{\Pi}^{+}_{X}(x)=\Pi_{X}(x,\infty),\qquad x>0.$$ We will work with the probabilities $\{\mathbb{P}_{x}:x\in
\mathbb{R}\}$ such that $\mathbb{P}_{x}(X_{0}=x)=1$ and $\mathbb{P}_{0}=\mathbb{P}$. The probabilities $%
\{\widehat{\mathbb{P}}_{x}:x\in \mathbb{R}\}$ will be defined in a similar sense for
the dual process, $-X$.

Denote by $\{(L_{t}^{-1},H_{t}):t\geq 0\}$ and $\{(\widehat{L}_{t}^{-1},%
\widehat{H}_{t}):t\geq 0\}$ the (possibly killed) bivariate subordinators
representing the ascending and descending ladder processes. Write $%
\kappa (\alpha ,\beta )$ and $\widehat{\kappa }(\alpha ,\beta )$ for their joint
Laplace exponents for $\alpha ,\beta \geq 0$. For convenience we will write
\[
\kappa (0,\beta )=q+\xi (\beta )=q+\mathrm{c}\beta +\int_{(0,\infty
)}(1-e^{-\beta x})\Pi _{H}({d}x),
\]%
where $q\geq 0$ is the killing rate of $H$ so that $q>0$ if and only if $%
\lim_{t\uparrow \infty }X_{t}=-\infty $, $\mathrm{c}\geq 0$ is the drift of $%
H$ and $\Pi _{H}$ is its jump measure. Similarly to $\Pi_X$ we shall define $\overline\Pi_H(x) = \Pi_H(x,\infty)$. The quantity $\xi $ is the  Laplace exponent of a true
subordinator. Similar notation will also be used for $%
\widehat{\kappa }(0,\beta )$ by replacing $q$, $\xi $, $\mathrm{c}$ and $%
\Pi _{H}$ by $\widehat{q}$, $\widehat{\xi }$, $\widehat{\mathrm{c}}$ and $%
\Pi _{\widehat{H}}$. Note that necessarily $q\widehat{q}=0.$ 

Associated with the ascending and descending ladder processes are the
bivariate renewal functions $V$ and $\widehat{V}$. The
former is defined by
\[
V({d}s, {d}x)=\int_{0}^{\infty }{d}t\cdot
\mathbb{P}(L_{t}^{-1}\in {d}s, H_{t}\in {d}x)
\]%
and taking double Laplace transforms shows that
\begin{equation}
\int_{0}^{\infty }\int_{0}^{\infty }e^{-\alpha s-\beta x}V({
d}x,{d}s)=\frac{1}{\kappa (\alpha ,\beta )}\quad\text{for }\alpha ,\beta
\geq 0  \label{doubleLT}
\end{equation}%
with a similar definition and relation holding for $\widehat{V}$.
These bivariate renewal measures are essentially the Green's measures of
the ascending and descending ladder processes. With an abuse of notation we shall also write $V({d}x)$ and $%
\widehat{V}({d}x)$ for the marginal measures $V(%
[0,\infty ), {d}x)$ and $\widehat{V}([0,\infty ), {d}x)
$ respectively. (Since we shall never use the marginals  $V(%
{d}s,[0,\infty ))$ and $\widehat{V}({d}s,[0,\infty ))
$ there should be no confusion). Note that local time at the maximum is defined only up to a
multiplicative constant. For this reason, the exponent $\kappa $ can only be
defined up to a multiplicative constant and hence the same is true of the
measure $V$ (and then obviously this argument applies to $\widehat{%
V}$).

Let
\[
\overline{X}_{t}:=\sup_{u\leq t}X_{u}
\]%
and  define for each $x\in \mathbb{R}$,
\[
\tau^+_{x}=\inf \{t>0:X_{t}>x\} 
\text{ and }
\overline{G}_{t}=\sup \{s< t:X_{s}=\overline{X}_{s}\}.
\]

\bigskip

A new identity was given in \cite{KD} for general L\'evy processes  which specifies  at first passage over a fixed level  the
quintuple law of: 
the time of first passage relative to the time of the last maximum at
first passage,
the time of the last maximum at first passage,
the overshoot at first passage,
the undershoot at first passage and
the undershoot of the  last maximum at first passage.
For sake of reference, the quintuple law is reproduced below.

\begin{teo}[Doney and Kyprianou \cite{KD}]
\label{quintupleDK}Suppose that $X$ is not a compound Poisson process. Then for each $x>0$ we have on $u>0$, $v\geq y$, $y\in \lbrack 0,x]$, $s,t\geq 0,$%
\begin{eqnarray*}
&&\mathbb{P}(\tau^+_{x} -\overline{G}_{\tau^+_{x}-}\in {d}t, \, 
\overline{G}_{\tau^+_{x}-}\in {d}s, \, 
X_{\tau^+_{x}}-x\in {d}u, \, 
x-X_{\tau^+_{x}-}\in {d}v,  \,
x-\overline{X}_{\tau^+_{x}-}\in
{d}y) \\
&&\hspace{3cm}={V}({d}s, x-{d}y)\widehat{V}(%
{d}t, {d}v-y)\Pi _{X}({d}u+v)
\end{eqnarray*}%
where the equality holds up to a multiplicative constant.
\end{teo}

\bigskip

Many of the results in this paper will follow as a consequence, either as an application or by similar reasoning,  of the above quintuple law. As mentioned earlier, we shall concentrate not only on the case of first passage above a fixed level, but also last passage above a fixed level. Additionally, limiting cases of such laws will also be on the agenda. Moreover, our reasoning permits us to deal with more than just L\'evy processes and we consider also L\'evy processes conditioned to stay positive as well as positive self-similar Markov processes. In all of the  cases we consider, depending on the setting, it will be possible to produce joint laws of anywhere between three to seven variables associated with the passage problem. We therefore collectively refer to our results as the $n$-tuple laws.

The principal motivation for this work is the wide variety of  applications that are connected to the first and last passage problem. In the theory of actuarial mathematics, the first passage problem is of fundamental interest with regard to the classical ruin problem and typically takes the form of the so called expected discounted penalty function. The latter is also known in the actuarial community as the Gerber-Shiu function following the first article \cite{MR1604928} of a series which has appeared in the actuarial literature. Recent literature, for example \cite{MR2146892}, also cites interest in the last passage problem within the context of ruin problems. In the setting of financial mathematics, the first passage problem is of interest in the pricing of barrier options in markets driven by L\'evy processes. In queueing theory  passage problems for L\'evy processes play a central role in understanding  the  trajectory of the workload during busy periods as well as in relation to buffers. Many optimal stopping strategies also turn out to boil down to first passage problems; a classical example of which being McKean's optimal stopping problem \cite{McKean}. It is not our purpose however to dwell on these applications as there is already much to say about the first and last passage problems as self contained problems.

Let us conclude this section by outlining the remaining presentation of the paper.  In the next section we present a family of three new quintuple laws. Firstly a quintuple law of a general class of L\'evy processes conditioned to stay positive and issued from the origin which concerns overshoots and undershoots at last passage above a level $x>0$. This quintuple law will follow from Theorem \ref{quintupleDK} as a natural consequence of the Tanaka path decomposition. Note that the latter  originates from the theory of conditioned random walks, but thanks to \cite{MR2126962} a version of the path decomposition is also available for L\'evy processes conditioned to stay positive. The aforementioned quintuple law at last passage may then be used to construct two further septuple laws at last passage. One for a L\'evy process conditioned to stay positive when issued from a positive position, and a second one for a L\'evy process without conditioning.
In Section \ref{asymptotics} we turn return to a family of results concerning asymptotic overshoot-undershoot triple laws at first passage which improve on recent contributions in the literature. The improvements lie with the increased number of variables in the joint laws as well as, in some cases, the possibility of negative jumps in addition to positive jumps. These are then used to establish asymptotic overshoot-undershoot triple laws at last passage for L\'evy processes and L\'evy processes conditioned to stay positive. Proofs are given in Section \ref{proofs}.
Next, in Section \ref{triplePSSMP} we use ideas behind asymptotic overshoot-undershoot triple laws to examine the stationary nature of overshoot-undershoot triple laws for positive self similar Markov processes issued from the origin.
Finally in Section \ref{egs} we consider some examples where the previously appearing identities become more explicit. Moreover we play off some of the conclusions from the previous two sections and conclude with some explicit computations for Stable and Lamperti-Stable processes, including some new, explicit identities.

\section{Quintuple and septuple laws at last passage}\label{5-7}

We start with our first result which describes the quintuple law at last passage for a L\'evy process conditioned to stay positive and issued from the origin. Henceforth we shall denote by $\mathbb{P}^\uparrow$ the law of $(X,\mathbb{P})$ conditioned to stay positive. This law can be constructed in several ways, see for example  \cite{CD, MR2320889}. However we will be specifically interested in  Tanaka's construction of the law $\mathbb{P}^\uparrow$ as described in \cite{MR2126962}. In the latter construction, which is valid for L\'evy processes which do not drift to $-\infty$ and for which $0$ is regular for $(-\infty,0)$, the excursions from $0$ of process $(X,\mathbb{P}^\uparrow)$ reflected at its future infimum  are those of $(X,\mathbb{P})$ reflected at its past supremum and time reversed. Moreover the closure of the set of zeros of the latter equals that of the former.

 Let $$\pox_{t}=\inf\{X_{s}: s\geq t\}$$ 
be the future infimum of $X$, 
$$\finfg_{t}=\sup\{s<t: X_{s}-\pox_{s}=0\} \qquad \finfd_{t}=\inf\{s>t: X_{s}-\pox_{t}=0 \}$$ 
are the  left and right end points of the excursion of $X$ from its future infimum straddling time $t$. 
Now define the last passage time $$U_{x}=\sup\{s\geq 0: X_{t}\leq x\}$$ and observe that $U_{x}$ can be the left or right extrema of an excursion interval of the process conditioned to stay positive reflected at its future infimum. However, if $x$ does not coincide with a point in $\{\pox_{t}, t\geq 0\}^{\rm{cl}}$ then $U_x$ corresponds to the left extrema of an excursion; in particular $U_x =\finfg_{U_x}$.

The quintuple law at last passage for L\'evy  processes conditioned to stay positive and issued from the origin, reads as follows.

\begin{teo}\label{quintupleKPR} Suppose that $X$ is a L\'evy process which does not drift to $-\infty$ and for which $0$ is regular for $(-\infty,0)$. For $s,t\geq 0$, $0< y\leq x$, $w\geq u>0$,
\begin{eqnarray*}
&&\mathbb{P}^\uparrow(
\finfd_{U_x}  - U_x \in {d}t, \,
U_x\in {d}s, \,
\underrightarrow{X}_{U_x} - x\in{d}u, \,
x-X_{U_x-}\in {d}y,\, %
X_{U_x}-x\in {d}w 
) \\
&&\hspace{3cm}=V({d}s, x-{d}y)\widehat{V}({d}t, 
w-{d}u)\Pi _{X}({d}w +y)
\end{eqnarray*}%
where the equality hold up to a multiplicative constant.
\end{teo}

\begin{proof} Suppose that 
$F:\re^5\to\re^+$ is a measurable and bounded function such that $F(\cdot, \cdot, \cdot, \cdot, 0) =0$. Thanks to Tanaka's path decomposition we may identify $\overline{G}_{\tau^+_x-} = U_x$, $\finfd_{U_{x}} =\tau^+_x$, $\pox_{U_{x}}  = X_{\tau^+_x}$, $X_{U_x - } = \overline{X}_{\tau^+_x -}$ and $X_{U_x} = \overline{X}_{\tau^+_x-} + X_{\tau^+_x} - X_{\tau^+_x -}$. Hence we may write directly the following identity,
\begin{equation*}
\begin{split}
&\mathbb{E}^\uparrow\left(F(\finfd_{U_{x}}-U_{x}, U_{x}, \pox_{U_{x}} - x, x - X_{U_{x}-}, X_{U_{x}}-x)\right)\\
&=\e\left(F\left(\tau^+_x -\overline{G}_{\tau^+_x-}, \overline{G}_{\tau^+_x-}, X_{\tau^+_x}-x, x-\overline{X}_{\tau^+_x-}, X_{{\tau^+_x}}-X_{{\tau^+_x}-}+\overline{X}_{\tau^+_x-}-x\right)\right)\\
&=\e\left(F\left(\tau^+_x -\overline{G}_{\tau^+_x-}, \overline{G}_{\tau^+_x-}, X_{\tau^+_x}-x, x-\overline{X}_{\tau^+_x-}, (X_{{\tau^+_x}}-x)+(x-X_{{\tau^+_x}-})-(x-\overline{X}_{\tau^+_x-})\right)\right)\\
&= \int_{(0,\infty)^5}{V}({d}s, x-{d}y)\widehat{V}(%
{d}t, {d}v-y)\Pi _{X}({d}u+v)F(s,t,u,y, u+v-y)\mathbf{1}_{\{y\leq x\wedge v\}}.
\end{split}
\end{equation*}
The result follows by a change of variables $w = u+v-y$ in the above integral.
Note in particular the assumption on $F$ allows us to exclude from the expectation  considerations corresponding to the L\'evy process creeping upwards;  equivalently that $x\in\{\pox_{t}, t\geq 0\}^{\rm{cl}}$.
\end{proof}

As a consequence of the quintuple law in Theorem \ref{quintupleKPR} we obtain the shortly following two corollaries which specify septuple laws at last passage for L\'evy processes  and for L\'evy processes conditioned to stay positive when issued from a positive position. In both corollaries we use the notation
\[
 \underline{G}_t = \sup\{s<t : X_s -\underline{X}_s =0 \}.
\]
where 
\[
\underline{X}%
_{t}:=\inf_{u\leq t}X_{u}.
\]
We also write $\mathbb{P}^\uparrow_z$ for the law of $X$ conditioned to stay positive when issued from $z>0$. It is known that the latter satisfies, for example, $\mathbb{P}^\uparrow_z(X_t \in dx) = \widehat{V}(z)^{-1}\widehat{V}(x)\mathbb{P}(X_t \in dx, \underline{X}_t >0)$ where $x>0$. Moreover, in the sense of weak convergence with respect to the Skorohod topology, $\lim_{z\downarrow 0}\mathbb{P}^\uparrow_z = \mathbb{P}^\uparrow$ when $0$ is regular for $(0,\infty)$. See Chaumont and Doney \cite{CD} for full details.

\begin{corollary}\label{7L1}
 Suppose that $X$ is a L\'evy process which does not drift to $-\infty$ and for which $0$ is regular for $(-\infty,0)$ as well as for $(0,\infty)$.
For $t,x,z> 0, s>r>0$, $0\leq v\leq z\wedge x$,  $0< y\leq x -v$, $w\geq u>0$,
\[
\begin{split}
&\mathbb{P}^\uparrow_z(
\underline{G}_{\infty} \in dr, \,
 \underline{X}_{\infty} \in dv, \,
\finfd_{U_x}  - U_x \in {d}t, \,
U_x\in {d}s, \,
\underrightarrow{X}_{U_x} - x\in{d}u, \,
x-X_{U_x-}\in {d}y,\, %
X_{U_x}-x\in {d}w 
) \\
&\hspace{2cm}=\widehat{V}(z)^{-1}\widehat{V}({d}r, z- dv)V({d}s-r, x-v-{d}y)\widehat{V}({d}t, 
w-{d}u)\Pi _{X}({d}w +y)
\end{split}
\]
where the equality holds up to a multiplicative constant. Moreover, in the particular case that $z>x$
\[
\mathbb{P}^\uparrow_z(U_x=0, \underline{G}_{\infty} \in dr, \,
 \underline{X}_{\infty} \in dv) = \widehat{V}(z)^{-1}\widehat{V}({d}r, z- dv)
\]
for $r>0$ and $v\in[x,z]$.
\end{corollary}


\begin{proof}
 The first part of the corollary is  a direct consequence of Millar's result for splitting a Markov process at its infimum; cf \cite{Mi1, Mi}. Indeed, according to the latter, the post infimum process is independent of the pre-infimum process and, relative  to the given space time point $(\underline{G}_{\infty},
 \underline{X}_{\infty})$ the post infimum process has the law of $\mathbb{P}^\uparrow.$  We should note that in Millar's description of the post-infimum process, the assumption that $0$ is regular for  $(0,\infty)$ means in particular that the process $X$ is right continuous at times which belong to the set $\{t>0: X_t = \underline{X}_t\}$. 

To compute the joint law of $(\underline{G}_{\infty},
 \underline{X}_{\infty})$, and thus complete the proof of the first part of the corollary, let $\mathbf{e}_q$ be an independent random variable which is exponentially distributed with rate $q>0$ 
 With the help of the compensation formula for the excursions of $X$ away from $\underline{X}$ we have that for $r>0$ and $v\in[0,z]$,
\begin{eqnarray}
\mathbb{P}^\uparrow_z(
\underline{G}_{\infty} \in dr, \,
 \underline{X}_{\infty} \in dv)& =& \lim_{q\downarrow 0}\widehat{V}(z)^{-1}\mathbb{E}_z(
\mathbf{1}_{\{\underline{G}_{\mathbf{e}_q} \in dr, \,
\underline{X}_{\mathbf{e}_q} \in dv\}} \widehat{V}(X_{\mathbf{e}_q}) \mathbf{1}_{\{\underline{X}_{\mathbf{e}_q}\geq 0\}}   ) \notag \\
&&=\widehat{V}(z)^{-1}\lim_{q\downarrow 0}\mathbb{E}(
\mathbf{1}_{\{\underline{G}_{\mathbf{e}_q} \in dr, \,
z+\underline{X}_{\mathbf{e}_q} \in dv\}} \widehat{V}(v+X_{\mathbf{e}_q} -\underline{X}_{\mathbf{e}_q} ) 
) .
\label{2cases}
\end{eqnarray}
When $X$ drifts to $+\infty$ the right hand side above is well defined as $\widehat{V}(\infty)<\infty$ and is equal to 
\[
\widehat{V}(\infty)\widehat{V}^{-1}(z)\mathbb{P}(
\underline{G}_{\infty} \in dr, \,
z+ \underline{X}_{\infty} \in dv)=\widehat{V}^{-1}(z)\widehat{V}(dr, z-dv).
\]
 Note that the last equality is consequence of the fact that the negative Wiener-Hopf factor takes the form 
\[
\mathbb{E}(e^{-\alpha\underline{G}_{\infty} -\beta \underline{X}_\infty})= \frac{\widehat{\kappa}(0,0)}{\widehat{\kappa}(\alpha, \beta)}
\]
the Laplace transform (\ref{doubleLT}) and that $\widehat{V}(\infty) = 1/\widehat{\kappa}(0,0)$.

Henceforth we assume that $X$ oscillates. Note that
\begin{equation}
\begin{split}
\mathbb{E}(
\mathbf{1}_{\{\underline{G}_{\mathbf{e}_q} \in dr, \,
z+\underline{X}_{\mathbf{e}_q} \in dv\}} \widehat{V}(v+&X_{\mathbf{e}_q} -\underline{X}_{\mathbf{e}_q} ) ) =\mathbb{E}\int_0^\infty qe^{-qt}\mathbf{1}_{\{\underline{G}_{t} \in dr, \,
z+\underline{X}_{t} \in dv\}} \mathbf{1}_{\{\underline{X}_t=X_t\}}   \widehat{V}(v)dt\\
&+\mathbb{E}\sum_{g}\mathbf{1}_{\{\underline{G}_{g-} \in dr, \,
z+\underline{X}_{g-} \in dv\}}
\int_g^{d_g} qe^{-qt} \widehat{V}(v+\epsilon_g(t))dt\label{I+II}
\end{split}
\end{equation}
where the sum is taken over all left end points, $g$, of excursions of $X$ from its infimum $\underline{X}$, with corresponding excursion and right end point denoted by   $\epsilon_g$ and $d_g$ respectively. 
Suppose that we call the two terms on the right hand side of (\ref{I+II}) $A_q$ and $B_q$.
Recalling that $X$ oscillates, 
we have
\[
\lim_{q\downarrow 0}A_q \leq  \lim_{q\downarrow 0}\widehat{V}(v) \mathbb{P}(
\underline{G}_{\mathbf{e}_q} \in dr, \,
z+ \underline{X}_{\mathbf{e}_q} \in dv)=0
\]
Appealing to the compensation formula for excursions we have
\[
B_q=\mathbb{E}\left(\int_0^\infty  \mathbf{1}_{\{\underline{G}_{s-} \in dr, \,
z+\underline{X}_{s-} \in dv\}} 
e^{-q s} d \widehat{L}_s
\right)
\underline{n}\left(\int_0^{\zeta} qe^{-qu} \widehat{V}(v+\epsilon(u))du \right)
\]
where $\epsilon$ is the generic excursion with life time $\zeta$  and $\underline{n}$ is the associated excursion measure.
After a change of variables $s\mapsto \widehat{L}^{-1}_t$, the first term on the right hand side above has a limit as $q\downarrow 0$ equal to $\widehat{V}(dr, z-dv)$. The second term on the other hand converges to a constant as $q\downarrow 0$ as we shall now explain. Note that it may be written in the form 
$
\underline{n}(\widehat{V}(v + \epsilon(\mathbf{e}_q	)) \mathbf{1}_{\{\mathbf{e}_q<\zeta \}})
$
which, on the one hand is lower bounded by $\underline{n}(\widehat{V}( \epsilon(\mathbf{e}_q	)) \mathbf{1}_{\{\mathbf{e}_q<\zeta \}})
$
and, on the other, is upper bounded 
by
$\underline{n}(\widehat{V}(v)  \mathbf{1}_{\{\mathbf{e}_q<\zeta \}})
+\underline{n}(\widehat{V}( \epsilon(\mathbf{e}_q	)) \mathbf{1}_{\{\mathbf{e}_q<\zeta \}}).
$
The latter bounds are respectively thanks to the monotonicity and subadditivity of the renewal function $\widehat{V}$. It is known (cf. \cite{CD}, \cite{silver}) that $\widehat{V}$ is harmonic in the sense that $\mathbb{E}(\widehat{V}(z+X_t)\mathbf{1}_{\{z+\underline{X}_t\ \geq 0\}})=\widehat{V}(z)$.
Appealing to the description of the excursion measure $\underline{n}$ in Theorem 3 in \cite{Ch96} 
we find that,
\[
\underline{n}(\widehat{V}(\epsilon_t), t<\zeta) = \mathbb{E}^\uparrow(1).\]

This in turn implies that $\underline{n}(\widehat{V}( \epsilon(\mathbf{e}_q	)) \mathbf{1}_{\{\mathbf{e}_q<\zeta \}})=1$. 
At the same time, since $X$ oscillates we have that  $\underline{n}(\zeta=\infty)=0$ and hence $\lim_{q\downarrow 0}\underline{n}(V(v)  \mathbf{1}_{\{\mathbf{e}_q<\zeta \}})=0$. It follows that $\lim_{q\downarrow 0}B_q$ is proportional to $\widehat{V}(dr, z-dv)$. Referring back to (\ref{2cases}) and (\ref{I+II}) this completes the proof of the first part of the corollary.

The proof of the second part of the corollary is a direct consequence of the joint law of $(\underline{G}_\infty, \underline{X}_\infty)$.
\end{proof}

\bigskip

\begin{remark}\rm
It is worth noting that contained in the proof of the above Corollary is a generalization to Chaumont's law of the global infimum of a L\`evy process conditioned to stay positive (cf. Theorem 1 of \cite{CD} for its most general form).  Namely that, under the conditions of the above corollary,  for $r\geq 0$ and $0\leq v\leq z$
\[
\mathbb{P}^\uparrow_z(
\underline{G}_{\infty} \in dr, \,
 \underline{X}_{\infty} \in dv) =\frac{\widehat{V}(dr, z-dv)}{ \widehat{V}(z)}.
\]
\end{remark}

\bigskip

\begin{corollary}\label{7L2}
Suppose that $X$ is a L\'evy process which drifts to $\infty$ and for which $0$ is regular for both $(-\infty,0)$ and $(0,\infty)$.
For $t,x,v> 0, s>r>0$,  $0\leq y<x+v$, $w\geq u>0$,
\[
\begin{split}
&\mathbb{P}(
\underline{G}_{\infty} \in dr, \,
- \underline{X}_{\infty} \in dv, \,
\finfd_{U_x}  - U_x \in {d}t, \,
U_x\in {d}s, \,
\underrightarrow{X}_{U_x} - x\in{d}u, \,
x-X_{U_x-}\in {d}y,\, %
X_{U_x}-x\in {d}w 
) \\
&\hspace{2cm}=\widehat{V}(\infty)^{-1}\widehat{V}({d}r, dv)V({d}s-r, x+v-{d}y)\widehat{V}({d}t, 
w - {d} u)\Pi _{X}({d}w +y)
\end{split}
\]
where the equality holds up to a multiplicative constant.
\end{corollary}
\begin{proof}
 The corollary is again a consequence of Millar's result for splitting a L\'evy process at its infimum. Specifically, the pre- and post-infimum processes are independent conditionally on the value of $(\underline{G}_{\infty}, -\underline{X}_{\infty})$ and relative to the latter space-time point, the law of the post-infimum process is $\mathbb{P}^\uparrow$. Moreover, a computation similar in the spirit to (but much easier than) the proof of the previous corollary shows that the law of the pair $(\underline{G}_{\infty}, -\underline{X}_{\infty})$ is given by $\widehat{V}({d}r, dv)$.
\end{proof}

\section{Asymptotic triple laws at first and last passage times}\label{asymptotics}
We begin this section by returning to asymptotic overshoot-undershoot laws of L\'evy processes at first and last passage. Related work on the forthcoming results can be found in  \cite{KK} and  \cite{Rsinai}. In both of the aforementioned articles, two dimensional asymptotic  overshoot-undershoot laws were obtained. Here we address the case of three dimensional overshoot-undershoot laws with the help of the following key observation.

For notational convenience frequently in this section we will denote 
the undershoots and overshoots at the first passage above a barrier as follows:
$$\mathcal{U}_{x}=x-\overline{X}_{\tau^+_{x}-},\quad \mathcal{V}_{x}=x-{X}_{\tau^+_{x}-},\quad \mathcal{O}_{x}={X}_{\tau^+_{x}}-x,\qquad x>0.$$

\begin{lemma}\label{lemmaOUshoot} For $u\leq x$, $v\geq u$, $w\geq0$ we have
\begin{equation*}
\begin{split}
\mathbb{P}\left(\mathcal{U}_{x}>u, \mathcal{V}_{x}>v, \mathcal{O}_{x}>w\right)
=\mathbb{P}\left(\mathcal{V}_{x-u}>v-u, \mathcal{O}_{x-u}>w+u\right).
\end{split}
\end{equation*}
\end{lemma}
\begin{proof}By virtue of the fact that $\tau^+_x$ is a first passage time recall that $x- \overline{X}_{\tau^+_x}\leq \mathcal{V}_x$.
On the event $\{\mathcal{U}_{x}>u, \mathcal{V}_{x}>v, \mathcal{O}_{x}>w\}$ the interval $[x-u, x+w]$ does not belong to the range of $\overline{X}$. This implies that $\mathcal{O}_{x-u}>u+w$ and $\mathcal{V}_{x-u}>v-u$. Conversely if the latter two inequalities hold, then we may again claim that the interval $[x-u, x+w]$ does not belong to the range of $\overline{X}$. Since $\mathcal{U}_{x-u}\in[0, \mathcal{V}_{x-u}]$ it follows that $\mathcal{U}_{x}>u, \mathcal{V}_{x}>v, \mathcal{O}_{x}>w$. The reader is encouraged to accompany the proof with a sketch at which point the proof becomes completely transparent.
\end{proof}

The above lemma tells us that studying the law of the triple $(x-\overline{X}_{\tau^+_{x}-}, \, x-{X}_{\tau^+_{x}-},\, X_{\tau^+_{x}}-x)$ is equivalent to studying the law of the pair $(x-{X}_{\tau^+_{x}-},X_{\tau^+_{x}}-x).$ This is a recurrent idea 
appearing in the proof of the theorems below.
We will also  make repeated use in the aforementioned proofs of an  important identity obtained by Vigon \cite{vigon} that relates $\Pi_H$, the L\'evy measure of the upward ladder height subordinator $H$,  with that of the L\'evy process $X$ and $\widehat{V}$, the potential measure of the downward ladder height subordinator $\widehat{H}.$ Specifically, defining $\overline{\Pi}_H(x)=\Pi_H(x,\infty)$, the identity states that 
 \begin{equation}\label{eqvigon}
\overline{\Pi}_{H}(r)=\int^{\infty}_{0}\widehat{V}(dl) \overline{\Pi}^+_{X}(l+r),\qquad r>0.
\end{equation} 


\begin{teo}\label{mainthmOUshoots} Let $X$ be a L\'evy process that does not drift to $-\infty.$
\begin{enumerate}
\item[(i)] Assume that the law of $X_{1}$ is not arithmetic.   
The triple $\left(x-\overline{X}_{\tau^+_{x}-},x-{X}_{\tau^+_{x}-},X_{\tau^+_{x}}-x\right)$ converges weakly as $x\rightarrow\infty$ towards a non-degenerate random variable if and only if $\mu_{+}:=\er(H_{1})<\infty$. In this case the limit law is given by
\begin{equation*}\label{oushoots0}
\begin{split}
&\hspace{-3cm}\lim_{x\to\infty}\mathbb{P}\left(x-\overline{X}_{\tau^+_{x}-}\in du, \, x-{X}_{\tau^+_{x}-}\in dv, \, X_{\tau^+_{x}}-x\in dw\right)\\
&\hspace{2cm}= \frac{1}{\mu_{+}}du\widehat{V}(dv-u)\Pi_{X}(dw+v)1_{\{v\geq u\geq 0, w>0\}}.
\end{split}
\end{equation*} In particular,
\begin{equation}\label{oushoots}
\begin{split}
&\hspace{-2cm}\lim_{x\to\infty}\mathbb{P}\left(x-\overline{X}_{\tau^+_{x}-}>u,\, x-{X}_{\tau^+_{x}-}>v,\, X_{\tau^+_{x}}-x>w\right)\\
&=\frac{1}{\mu_{+}}\int^{v-u}_{0}dy\int_{[y,\infty)}\widehat{V}(dl)\overline{\Pi}_X^+(w+l+v-y)+\frac{1}{\mu_{+}}\int^\infty_{v}dy\overline{\Pi}_{H}(w+y)
\end{split}
\end{equation} where $0\leq u\leq v,$ $w\geq 0.$
  
\item[(ii)] If there exists a non-decreasing function $b:(0,\infty)\to(0,\infty)$ such that $X_{t}/b(t)$  converges weakly , as $t\to\infty$,  towards a strictly stable random variable with index $\alpha\in(0,2),$ and positivity parameter $\rho\in(0,1),$ then 
\begin{equation*}
\begin{split}
&\hspace{-2cm}\lim_{x\rightarrow\infty}\mathbb{P}\left(\frac{x-\overline{X}_{\tau^+_{x}-}}{x}\in du,\frac{x-{X}_{\tau^+_{x}-}}{x}\in dv ,\frac{X_{\tau^+_{x}}-x}{x}\in dw\right) \\
&=\frac{\sin(\alpha\rho\pi)}{\pi}\frac{\Gamma(\alpha+1)}{\Gamma(\alpha\rho)\Gamma(\alpha(1-\rho))}\frac{(1-u)^{\alpha\rho-1}(v-u)^{\alpha(1-\rho)-1}}{(v+w)^{1+\alpha}}\, du\,dv\,dw,
\end{split}
\end{equation*}
for $0\leq u\leq 1,$ $v\geq u,$ and  $w>0.$ 

\item[(iii)] Assume that $X$ oscillates and that the mean of $\widehat{H}_{1}$ is finite. Suppose moreover that $ \overline{\Pi}^+_{X}:= \Pi_X^+(x,\infty)$ is regularly varying at $\infty$ with index $-1-\alpha$ for some $\alpha\in(0,1).$
Then
\begin{equation*}
\begin{split}
&\hspace{-2cm}\hspace{-2cm}\lim_{x\rightarrow\infty}\mathbb{P}\left(\frac{x-\overline{X}_{\tau^+_{x}-}}{x}\in du,\frac{x-{X}_{\tau^+_{x}-}}{x}\in dv ,\frac{X_{\tau^+_{x}}-x}{x}\in dw\right) \\
&=\frac{\alpha(1+\alpha)}{\Gamma(\alpha)\Gamma(1-\alpha)}\frac{1}{(1-u)^{1-\alpha}(v+w)^{2+\alpha}}\,du\,dv\,dw,
\end{split}
\end{equation*}
for $0< u<1$, $v\geq u$ and $w>0$.

\item[(iv)] Assume that $X$ drifts to $\infty$ and that $ \overline{\Pi}^+_{X}$ is regularly varying at $\infty$ with index $-\alpha$ for some $\alpha\in(0,1).$ Then
\begin{equation*}
\lim_{x\to\infty}\mathbb{P}\left(\frac{x-{X}_{\tau^+_{x}-}}{x}\in dv ,\frac{X_{\tau^+_{x}}-x}{x}\in dw\right) \\
=\frac{\alpha}{\Gamma(\alpha)\Gamma(1-\alpha)}\frac{1}{(1-v)^{1-\alpha}(v+w)^{\alpha+1}}\,dv\,dw,
\end{equation*}
for $w>0$ and  $0<v<1$
Furthermore, $$\frac{\overline{X}_{\tau^+_x-}-X_{\tau^+_x-}}{x}=\frac{x-{X}_{\tau^+_{x}-}}{x}-\frac{x-\overline{X}_{\tau^+_{x}-}}{x}\xrightarrow[x\to\infty]{\mathbb{P}}0.$$
\end{enumerate}
\end{teo}

\begin{remark}\rm It is important to mention that the assumptions in Theorem~\ref{mainthmOUshoots} can be verified using only the characteristics of the underlying L\'evy process $X .$
According to a result due to Chow~\cite{MR834219} necessary and
sufficient conditions on $X$ to be such that
$\er(\widehat{H}_1)<\infty,$ are either $0<\er(-X_1)\leq
\er|X_1|<\infty$ or $0=\er(-X_1)<\er|X_1|<\infty$ and
$$\int_{[1,\infty)}\left(\frac{x\overline{\Pi}_X^-(x)}{1+\int^x_0 d y\int^{\infty}_y\overline{\Pi}_X^+(z) d z}\right) d x<\infty \quad \text{with}\ \overline{\Pi}_X^-(x)=\Pi_X(-\infty,-x),\ x>0.$$
Observe that under such assumptions the L\'evy process $X$ does
not drift to $\infty,$ i.e. $\liminf_{t\to\infty}X_t=-\infty,$
$\pr$--a.s. Kesten and
Erickson's criteria state that $X$ drift to $\infty$ if and only
if
$$\int_{(-\infty,-1)}\left(\frac{|y|}{\overline{\Pi}_X^+(1)+\int^{|y|}_1\overline{\Pi}_X^+(z) d z}\right)\Pi_X( d y)<\infty=\int^{\infty}_1\overline{\Pi}_X^+(x) d x
\quad\text{or}\quad 0<\er(X_1)\leq \er|X_1|<\infty,$$ cf.
\cite{MR0266315} and \cite{MR0336806}. (In fact, Chow, Kesten
and Erickson proved the results above for random walks, its translation for real valued L\'evy processes can be found in \cite{doneymaller} and \cite{vigt}.) Moreover, a sufficient condition in terms of the tail L\'evy measure of $X$ for the hypothesis in (ii) in Theorem~\ref{mainthmOUshoots} to be satisfied can be found in Lemma 5 in \cite{Rsinai}.
\end{remark}

\begin{remark}\rm 
Under suitable hypotheses, which can be found in  \cite{Rsinai}, and using very similar methods, it is possible to establish an analogue of the latter result when $x\to 0.$ As this article is rather long already, and for sake of conciseness, we have chosen not to include a proof nor a statement. 
\end{remark}

A simple but interesting consequence of Theorem~\ref{quintupleKPR}  and Theorem~\ref{mainthmOUshoots}  are the following asymptotic triple laws for the overshoot and undershoot at the last passage above a barrier of a L\'evy process conditioned to stay positive. We just state the result under the assumption that the process starts from $0,$ although, thanks to a simple application of Corollary~\ref{7L1}, a similar result holds when the process starts from a strictly positive position. Moreover, using Corollary~\ref{7L2} it is also possible to establish from the following Corollary the analogous result for the asymptotic law for the overshoot and undershoot at the last passage above a barrier of a L\'evy process. We leave the details to the interested reader.       

\begin{corollary}\label{cor 2}
Suppose that $X$ is a L\'evy process which does not drift to $-\infty$ and for which $0$ is regular for $(-\infty,0)$ and $(0,\infty)$. If the assumptions of Theorem \ref{mainthmOUshoots} (i) are satisfied then asymptotic three dimensional law of overshoots and undershoot of a L\'evy process conditioned to stay positive is given by
\begin{equation*}
\begin{split}
&\lim_{x\to\infty}\mathbb{P}^\uparrow( x-X_{U_x-}\in {d}y,\,  X_{U_x}-x\in {d}u,\,  \underrightarrow{X}_{U_x} - x\in{d}w
)\\
&=\frac{{d}y}{\mu_{+}}\Pi _{X}({d}u+y)\widehat{V}(u-{d}w),
\end{split}
\end{equation*}
for $y>0,$ $0\le w\leq u$.

If, respectively,  the assumptions in (ii),  (iii) or (iv) in Theorem~\ref{mainthmOUshoots} are satisfied then 
\begin{equation*}
\begin{split}
&\lim_{x\to\infty}\mathbb{P}^\uparrow\left( \frac{x-X_{U_x-}}{x}\in {d}y,\,  \frac{X_{U_x}-x}{x}\in {d}u,\,  \frac{\underrightarrow{X}_{U_x} - x}{x}\in{d}w
\right)\\
&\hspace{2cm}=\begin{cases}  
\displaystyle\frac{\sin(\alpha\rho\pi)\Gamma(\alpha+1)}{\pi\Gamma(\alpha\rho)\Gamma(\alpha(1-\rho))}\frac{(1-y)^{\alpha\rho-1}(u-w)^{\alpha(1-\rho)-1}}{(u+y)^{1+\alpha}}\,dy\,du\,dw, & \text{in case (ii)}\\
\displaystyle \frac{\alpha(1+\alpha)}{\Gamma(\alpha)\Gamma(1-\alpha)}\frac{1}{(1-y)^{1-\alpha}(u+y)^{2+\alpha}}\,dy\,du\,dw, & \text{in case (iii)}\\
\displaystyle \frac{\alpha}{\Gamma(\alpha)\Gamma(1-\alpha)}\frac{1}{(1-y)^{1-\alpha}(u+y)^{1+\alpha}}\,dy\,du\,\delta_{u}(dw), & \text{in case (iv)}
\end{cases}
\end{split}
\end{equation*}
for $0<y<1, u\geq w>0.$
\end{corollary}


Theorem~\ref{mainthmOUshoots} excludes the case of a L\'evy process that drifts to $-\infty.$ In that case the process has a strictly positive probability of never crossing a given positive barrier and hence the overshoots and undershoot take the value $\infty$ on that event. Nevertheless, it is possible to establish similar results  to those established in Theorem~\ref{mainthmOUshoots} conditionally on the event that the process reaches the level. That is the purpose of the following results which are in turn a generalization of the results in \cite{KK} where the asymptotic behaviour of the overshoot and undershoot of a spectrally positive L\'evy process has been studied. We will assume that the L\'evy measure is subexponential. Analogous results for three dimensional undershoot and overshoot laws in the case where the underlying L\'evy process has a close to exponential L\'evy measure have been obtained in \cite{KD}.

Recall that a probability distribution function $F$ over $[0,\infty)$ is said to be subexponential if the tail distribution, $\overline{F}(x):=1-F(x),$ $x\in\re,$ satisfies that $\overline{F}(x)>0,$ $x>0,$ and 
\begin{gather}
\lim_{x\to\infty}\frac{\overline{F^{*2}}(x)}{\overline{F}(x)}=2.
\end{gather}
We will say that the L\'evy measure $\Pi^+_{X}:=\Pi_X|_{[0,\infty)}$ is sub-exponential if the distribution of the probability measure  $ \Pi_{X}[1,\infty)^{-1}\Pi_{X}|_{[1,\infty)}$ is subexponential.  We are interested in particular in two special cases of subexponential distributions: those which are regularly varying  and those in the domain of attraction of a Gumbel distribution.
\begin{teo}\label{mainthmOUshoots:bis}
Let $X$ be a real valued L\'evy process drifting to $-\infty,$ with subexponential right tail L\'evy measure, and such that the mean of $\widehat{H}_{1}$ is finite and put $\mu_{-}=\er(\widehat{H}_{1})$.
\begin{enumerate}
\item[(i)] If $ \overline{\Pi}^+_{X}$ is regularly varying at $\infty$ with index $-1-\alpha$ for some $\alpha\in(0,1)$ then
$$
\lim_{x\to\infty}\mathbb{P}\left(\left.\frac{x-\overline{X}_{\tau^+_{x}-}}{x} \in du, \, \frac{-X_{\tau^+_x-}}{x}>v, \, \frac{X_{\tau^+_{x}}-x}{x}>w \right| \tau^+_x<\infty\right)=\frac{1}{\left(1+v+w\right)^{\alpha}}\delta_1(du)
$$ 
for $u,v,w\geq 0$. Note in particular (with regard to the first element of the triple) that the limiting distribution is concentrated on $\{1\}\times[0,\infty)^2$.

\item[(ii)] Assume $\overline{\Pi}^+_{X}$ is in the maximum domain of attraction of the Gumbel distribution. Let $a:\re^+\to(0,\infty)$ be a continuous and differentiable function such that 
\[
a(x)\sim \int^\infty_{x}dy\overline{\Pi}^+_{X}(y)/\overline{\Pi}^+_{X}(x)
\]
  and $a'(x)\to 0$ as $x\to\infty.$ Then  
  $$
\lim_{x\to\infty}\mathbb{P}\left(\left.\frac{x-\overline{X}_{\tau^+_{x}-}}{x} \in du, \, \frac{-X_{\tau^+_x-}}{a(x)}>v, \, \frac{X_{\tau^+_{x}}-x}{a(x)}>w\right| \tau^+_x<\infty\right)=e^{-(v+w)}\delta_1(du),
$$ 
for $u,v,w\geq 0$. Note again in particular (with regard to the first element of the triple) that the limiting distribution is concentrated on $\{1\}\times[0,\infty)^2$. 
\end{enumerate}
\end{teo}

\section{Proof of Theorems \ref{mainthmOUshoots} and \ref{mainthmOUshoots:bis}}\label{proofs}

\begin{proof}[Proof of (i) in Theorem \ref{mainthmOUshoots}]
We will prove that the Laplace transform of the triple law of overshoot and undershoots converges pointwise which is enough for proving the claimed weak convergence. From the quintuple law in Theorem~\ref{quintupleDK} we obtain that the Laplace transform of the undershoots and overshoot of $X$ can be written as
\begin{equation*}
\begin{split}
&\e\left(\exp\{-\theta_{1}(x-\overline{X}_{\tau^+_{x}-})-\theta_{2}(x-{X}_{\tau^+_{x}-})-\theta_{3}(X_{\tau^+_{x}}-x)\}\right)\\
&=\int^x_{0}V(dy)\int^\infty_{0}\widehat{V}(dl)\int_{z>x-y+l}\Pi_{X}(dz)\exp\left\{-\theta_{1}(x-y)-\theta_{2}(x-y+l)-\theta_{3}(z-(l+x-y))\right\}\\
&=\int^x_{0}V(dy)e^{-(\theta_{1}+\theta_{2})(x-y)}\int^\infty_{0}\widehat{V}(dl)e^{-\theta_{2}l}\int_{z>x-y+l}\Pi_{X}(dz)e^{-\theta_{3}(z-(l+x-y))},
\end{split}
\end{equation*}
for $x>0$ and $\theta_{1},\theta_{2},\theta_{3}\geq 0.$ The proof will follow from an application of the version of the key renewal theorem appearing in Theorem 5.2.6 of \cite{jagers} and the remark following it, applied to the renewal measure $V(dy)$  and the function  \begin{equation}\label{LT1}r\mapsto e^{-(\theta_{1}+\theta_{2})r}\int^\infty_{0}\widehat{V}(dl)e^{-\theta_{2}l}\int_{z>r+l}\Pi_{X}(dz)e^{-\theta_{3}(z-(l+r))}, \qquad r>0.\end{equation} 
To this end observe that the measure $V(dy)$ is the renewal measure associated with the probability measure $\p(H_{\mathbf{e}}\in dy),$ where $\mathbf{e}$ is an independent exponential random variable with unit mean. An easy calculation using Laplace transforms shows that the random variable $H_{\mathbf{e}}$ is non-arithmetic, because $X$ has the same property. Moreover,   $\e(H_{\mathbf{e}})=\mu_{+}<\infty;$ and  the function defined in (\ref{LT1}) is  bounded  above by the decreasing and integrable function $$r\mapsto e^{-(\theta_{1}+\theta_{2})r}\int^\infty_{0}\widehat{V}(dl) \overline{\Pi}^+_{X}(l+r),\qquad r>0.$$ Note that  the integrability of this function  follows from (\ref{eqvigon}) and the fact that, by assumption, the mean of $H_{1}$ is finite or equivalently $\overline{\Pi}_{H}$ is integrable. 

We can hence apply the version of the renewal theorem appearing in \cite{jagers}~Theorem 5.2.6 and the remark following it, to deduce that for $\theta_{1},\theta_{2}, \theta_{3}\geq 0$,
\begin{equation*}
\begin{split}
&\lim_{x\to\infty}\e\left(\exp\{-\theta_{1}(x-\overline{X}_{\tau^+_{x}-})-\theta_{2}(x-{X}_{\tau^+_{x}-})-\theta_{3}(X_{\tau^+_{x}}-x)\}\right)\\&=\frac{1}{\mu_{+}}\int^{\infty}_{0}dy\int^\infty_{0}\widehat{V}(dl)\int_{z>y+l}\Pi_{X}(dz)\exp\left\{-\theta_{1}y-\theta_{2}(y+l)-\theta_{3}(z-(l+y))\right\},
\end{split}
\end{equation*}
which implies the result. The formula for the asymptotic distribution of the three dimensional law of overshoot and undershoots is obtained using elementary arguments. 

To establish the other direction of the proof, we observe that if the triple of random variables $(x-{\overline{X}}_{\tau^+_{x}-},x-{X}_{\tau^+_{x}-},X_{\tau^+_{x}}-x)$ converges weakly as $x\to\infty$ to a non-degenerate random variable then $X_{\tau^+_{x}}-x$ also converges weakly to a non-degenerate random variable. It is known that this implies that the mean of $H_{1}$ is finite, see e.g. Theorem 8 and Lemma 7 in \cite{doneymaller}.
\end{proof}

\bigskip

\begin{proof}[Proof of (ii) in Theorem \ref{mainthmOUshoots}]
The proof of this result is essentially an extension  of the method of proof used in  Theorem 2 in \cite{Rsinai} so we will just provide an sketch of proof. Let $X^{(r)}$ be the L\'evy process defined by $(X_{rt}/b(r),t\geq 0)$ and $\widetilde{X}=(\widetilde{X}_{t}, t\geq 0)$ be a strictly $\alpha$-stable L\'evy process. 

By assumption $X^{(r)}_{1}$ converges weakly to $\widetilde{X}_{1},$ and it is well known that this implies that the process $X^{(r)}$ converges weakly towards $\widetilde{X}.$ The undershoot and overshoot of $X^{(r)}$ are such that 
$$\left(\mathcal{V}^{(r)}_{1},\mathcal{O}^{(r)}_{1}\right)=\left(\frac{\mathcal{V}_{b(r)}}{b(r)},\frac{\mathcal{O}_{b(r)}}{b(r)}\right),\qquad r>0,$$ in the obvious notation. By  Theorem 13.6.4 in \cite{whitt} it follows that $(\mathcal{V}^{(r)}_{1},\mathcal{O}^{(r)}_{1})$ converges weakly towards the undershoot and overshoot $(\widetilde{\mathcal{V}}_{1},\widetilde{\mathcal{O}}_{1})$ of $\widetilde{X}$ at the level $1.$ Which implies that
$$\left(\frac{\mathcal{V}_{b(r)}}{b(r)},\frac{\mathcal{O}_{b(r)}}{b(r)}\right)\xrightarrow[r\to\infty]{D} \left(\widetilde{\mathcal{V}}_{1},\widetilde{\mathcal{O}}_{1}\right).$$

Next we appeal to an argument similar to the one used in the proof of Theorem 2 in \cite{Rsinai} to justify that 
 $$\left(\frac{x-{X}_{\tau^+_{x}-}}{x},\frac{X_{\tau^+_{x}}-x}{x}\right)\xrightarrow[r\to\infty]{D} \left(\widetilde{\mathcal{V}}_{1},\widetilde{\mathcal{O}}_{1}\right).$$
The proof is based on the fact that the asymptotic inverse of $b,$ say $b^{\leftarrow},$ is such that $b(b^{\leftarrow}(x))\sim x$ as $x\to\infty.$ Now using Lemma~\ref{lemmaOUshoot} we get that for $u\in[0,1[,$ $v\geq u,$ $w>0,$
\begin{equation}\label{conv3from2}
\begin{split}
&\hspace{-2cm}\lim_{x\to\infty}\mathbb{P}\left(\frac{x-\overline{X}_{\tau^+_{x}-}}{x}>u, \frac{x-{X}_{\tau^+_{x}-}}{x}>v, \frac{X_{\tau^+_{x}}-x}{x}>w \right)\\
&\hspace{4cm}=\lim_{z\to\infty}\mathbb{P}\left(\frac{\mathcal{V}_{z}}{z}>\frac{v-u}{1-u}, \frac{\mathcal{O}_{z}}{z}>\frac{w+u}{1-u}\right)\\
&\hspace{4cm}=\mathbb{P}\left(\widetilde{\mathcal{V}}_{1}>\frac{v-u}{1-u}, \widetilde{\mathcal{O}}_{1}>\frac{w+u}{1-u}\right).
\end{split}
\end{equation}
To conclude the proof and for sake of reference we quote, from \cite{KD}, the formula for the law of the undershoots-overshoot at level $1$ for a stable process with index $\alpha\in(0,2)$ and positivity parameter $\rho\in(0,1),$
\begin{equation}\label{eq:SOU3L}
\begin{split}
&\mathbb{P}\left(\widetilde{\mathcal{U}}_{1}\in du, \widetilde{\mathcal{V}}_{1}\in dv, \widetilde{\mathcal{O}}_{1}\in dw\right)\\
&\hspace{3cm}=\frac{\sin(\alpha\rho\pi)}{\pi}\frac{\Gamma(\alpha+1)}{\Gamma(\alpha\rho)\Gamma(\alpha(1-\rho))}\frac{(1-u)^{\alpha\rho-1}(v-u)^{\alpha(1-\rho)-1}}{(v+w)^{1+\alpha}}\, du\,dv\,dw,
\end{split}
\end{equation}
for $0\leq u\leq 1,$ $v\geq u,$ and  $w>0.$ 
From (\ref{conv3from2}), (\ref{eq:SOU3L}) and elementary calculations we infer the required weak convergence.
\end{proof}

\bigskip

\begin{proof}[Proof of (iii) in Theorem \ref{mainthmOUshoots}] 
The proof of (iii) and (iv) are based on the following basic identity which can easily be extracted from  from the quintuple law in Theorem \ref{quintupleDK}. For $ v,w\geq 0$
\begin{equation}\label{eq:OUbasic}
\mathbb{P}\left(x-{X}_{\tau^+_{x}-}>v, X_{\tau^+_{x}}-x>w\right)=\int^x_{0}V(dy)\int_{[(v-(x-y))\vee 0,\infty)}\widehat{V}(dl) \overline{\Pi}^+_{X}(w+l+(x-y)).
\end{equation}
From the basic identity (\ref{eq:OUbasic}) we have that for $a,b>0,$
\begin{equation}\label{eq:OUregvar}
\begin{split}
&\hspace{-2cm}\mathbb{P}\left(x-{X}_{\tau^+_{x}-}>xb, X_{\tau^+_{x}}-x>ax\right)\\
&=1_{\{0< b\leq 1\}}\int^{(1-b)}_{0}V(xdy)\int_{[0,\infty)}\widehat{V}(dl) \overline{\Pi}^+_{X}(ax+l+x(1-y))\\
&\quad+ 1_{\{0< b\leq 1\}}\int^{1}_{(1-b)}V(xdy)\int_{\left[x\left(b-(1-y)\right),\infty\right)}\widehat{V}(dl) \overline{\Pi}^+_{X}(ax+l+x(1-y))\\
&\quad+1_{\{b>1\}}\int^{1}_{0}V(xdy)\int_{\left[x\left(b-(1-y)\right),\infty\right)}\widehat{V}(dl) \overline{\Pi}^+_{X}(ax+l+x(1-y)).
\end{split}
\end{equation}
Thanks to Theorem 3 (a) in \cite{Rsinai}, under the assumptions in (iii), we have the estimate 
\begin{equation}\label{tailladder1}
\overline{\Pi}_{H}(x)\sim\frac{1}{\mu_{-}}\int^\infty_{x} \overline{\Pi}^+_{X}(z)dz,\qquad x\to\infty,\end{equation} where $\mu_{-}=\er(\widehat{H}_{1})<\infty.$ By Karamata's theorem, see e.g.  Chapter 1 of \cite{BGT}, and the assumption that $\overline{\Pi}^{+}_{X}$ is regularly varying at infinity with an index $-(1+\alpha),$ for $\alpha\in(0,1),$ it follows that $\overline{\Pi}_{H}$ is regularly varying at infinity with index $-\alpha.$ By Proposition 1.5 in \cite{bertoinsub} we have that $$\overline{\Pi}_{H}(x)V[0,ax]\xrightarrow[x\to\infty]{}\frac{a^{\alpha}}{\alpha\Gamma(\alpha)\Gamma(1-\alpha)},\qquad a>0.$$ 
This implies the weak convergence of measures  
\begin{equation}\label{LT:001}\overline{\Pi}_{H}(x)V(xdy)1_{\{x\in(0,1]\}}\xrightarrow[x\to\infty]{\text{weakly}}\frac{y^{\alpha-1}}{\Gamma(\alpha)\Gamma(1-\alpha)}1_{\{y\in(0,1]\}}dy\end{equation}
We claim that the asymptotic results in (\ref{tailladder1}) and (\ref{LT:001}) imply also that
\begin{equation}
\frac{1}{\overline{\Pi}_{H}(x)}\int_{\left[x\left(b-(1-y)\right),\infty\right)}\widehat{V}(dl) \overline{\Pi}^+_{X}(ax+l+x(1-y))\xrightarrow[x\to\infty]{}(a+b)^{-\alpha},
\label{uniformly}
\end{equation} 
uniformly in $(a+b)\in[t,\infty),$ for every $t>0.$ To see this, we apply the renewal theorem  to $\widehat{V}$ and use the estimate (\ref{tailladder1}). More precisely, the renewal theorem implies that, for $h>0$ and $\epsilon>0$, there exists a $t_0>0$ such that $$\left|\widehat{V}[t,t+h)-\frac{h}{\mu_{-}}\right|<\frac{\epsilon h}{\mu_{-}},\qquad \forall t\geq t_0.$$ Hence,  for every $\epsilon,$ $h>0$ and $x$ large enough 
\begin{equation*}
\begin{split}
&\int_{\left[x\left(b-(1-y)\right),\infty\right)}\widehat{V}(dl) \overline{\Pi}^+_{X}(ax+l+x(1-y))\\
&=\int_{\left[x(b-(1-y)),\infty\right)}\widehat{V}(dl) \overline{\Pi}^+_{X}(l-x(b-(1-y))+x(a+b))\\
&=\sum^{\infty}_{n=0}\int_{\left[x(b-(1-y))+nh,x(b-(1-y))+(n+1)h\right)}\widehat{V}(dl) \overline{\Pi}^+_{X}(l-x(b-(1-y))+x(a+b))\\
&\leq\sum^{\infty}_{n=0}\widehat{V}\left[x(b-(1-y))+nh,x(b-(1-y))+(n+1)h\right) \overline{\Pi}^+_{X}(nh+x(a+b))\\
&\leq\frac{(1+\epsilon)h}{\mu_{-}}\sum^\infty_{n=0} \overline{\Pi}^+_{X}(nh+x(a+b))\\
&\leq \frac{(1+\epsilon)}{\mu_{-}}\int^\infty_{0} \overline{\Pi}^+_{X}(z+x(a+b))dz\\ 
&\sim (1+\epsilon)\overline{\Pi}_{H}(x)(a+b)^{-\alpha}\qquad \text{ as }x\rightarrow \infty,
\end{split}
\end{equation*}
where the final estimate follows from (\ref{tailladder1}) and the regular variation of $\overline{\Pi}_H$.
This implies that $$\limsup_{x\to\infty}\frac{1}{\overline{\Pi}_{H}(x)}\int_{\left[x(b-(1-y)),\infty\right)}\widehat{V}(dl)\overline{\Pi}_X^+(l-x(b-(1-y))+x(a+b))\leq(a+b)^{-\alpha},$$
for all $a+b>0.$ The analogue estimate for the limit inferior is obtained  in a similar way thus justifying (\ref{uniformly}), but not uniformly in $(a+b)>t$ for every $t>0$. The aforesaid uniformity in $(a+b)$ follows from the fact that as $\overline{\Pi}_{H}$ is regularly varying at infinity then \begin{equation}\label{LT:00}\lim_{x\to\infty}\frac{\overline{\Pi}_{H}(cx)}{\overline{\Pi}_{H}(x)}=c^{-\alpha}, \text{uniformly in } c\in[t,\infty),\end{equation} for each $t>0,$ see e.g. Theorem 1.5.2 on p. 22 of  \cite{BGT}.  
Using the weak convergence in (\ref{LT:001}) and the uniformity in (\ref{LT:00}) we get that for $0<  b\leq 1,$ $a>0,$ the first term in equation (\ref{eq:OUregvar}) tends as $x\to\infty$ towards 
\begin{equation*}
\begin{split}
\int^{(1-b)}_{0}V(xdy)\int_{[0,\infty)}&\widehat{V}(dl) \overline{\Pi}^+_{X}(ax+l+x(1-y))\\&=\int^{(1-b)}_{0}\overline{\Pi}_{H}(x)V(xdy)\frac{\overline{\Pi}_{H}(ax+x(1-y))}{\overline{\Pi}_{H}(x)}\\
&\xrightarrow[x\to\infty]{}\frac{1}{\Gamma(1-\alpha)\Gamma(\alpha)}\int^{(1-b)}_{0} y^{\alpha-1}\frac{1}{\left(a+(1-y)\right)^{\alpha}}dy
\end{split}
\end{equation*}
In addition, arguing as above, using instead of (\ref{LT:00}) the property  (\ref{uniformly}), we may deal with  the second and third terms in equation (\ref{eq:OUregvar}) as $x\to\infty$ and obtain for $0< b\leq 1$
\begin{equation*}
\begin{split}
\int^1_{1-b}V(xdy)\int_{\left[x\left(b-(1-y)\right),\infty\right)}&\widehat{V}(dl) \overline{\Pi}^+_{X}(ax+l+x(1-y))\\
&\hspace{2cm}\xrightarrow[x\to\infty]{}\frac{1}{\Gamma(1-\alpha)\Gamma(\alpha)}\int^{1}_{(1-b)}y^{\alpha-1}\frac{1}{(a+b)^{\alpha}}dy
\end{split}
\end{equation*}
and for $b>1$
\begin{equation*}
\begin{split}
\int^{1}_{0}V(xdy)\int_{\left[x\left(b-(1-y)\right),\infty\right)}&\widehat{V}(dl) \overline{\Pi}^+_{X}(ax+l+x(1-y))\\
&\hspace{1cm}\xrightarrow[x\to\infty]{}\frac{1}{\Gamma(1-\alpha)\Gamma(\alpha)}\int^{1}_{0}y^{\alpha-1}\frac{1}{(a+b)^{\alpha}}dy,
\end{split}
\end{equation*}
respectively. Putting the three terms together back in (\ref{eq:OUregvar}) we get 
\begin{equation*}
\begin{split}
&\hspace{-3cm}\lim_{x\to\infty}\mathbb{P}\left(\frac{x-{X}_{\tau^+_{x}-}}{x}>b, \frac{X_{\tau^+_{x}}-x}{x}>a\right)\\
&=\frac{1}{\Gamma(1-\alpha)\Gamma(\alpha)}\int^{1}_{0} y^{\alpha-1}\frac{1}{\left(a+(1-y)\vee b\right)^{\alpha}}dy,
\end{split}
\end{equation*}
for $a,b>0.$ Taking derivatives we deduce that the weak limit, as $x\to\infty,$ of the law of the couple  $\displaystyle \left(\frac{x-{X}_{\tau^+_{x}-}}{x}, \frac{X_{\tau^+_{x}}-x}{x}\right),$ has a density given by 
\begin{equation}\label{density(iii)}
\frac{(1+\alpha)}{\Gamma(\alpha)\Gamma(1-\alpha)}\frac{\left(1-(1-v)^{\alpha}_{+}\right)}{(v+w)^{\alpha+2}},\quad v,w>0,
\end{equation}
where $z_{+}= \max\{z,0\}$.
Lemma \ref{lemmaOUshoot} and the identity~(\ref{conv3from2}) allows us to infer the weak convergence of the triplet $\left(\frac{x-\overline{X}_{\tau^+_{x}-}}{x}, \frac{x-{X}_{\tau^+_{x}-}}{x}, \frac{X_{\tau^+_{x}}-x}{x}\right).$ 
Using (\ref{density(iii)}) we deduce the form of the density for the asymptotic law for the overshoot and undershoots.
\end{proof}

\bigskip

\begin{proof}[Proof of (iv) in Theorem \ref{mainthmOUshoots}] The proof is based on the identity (\ref{eq:OUregvar}) and the fact, proved in Theorem 3 (b) and Corollary 2 (b) in \cite{Rsinai}, that under our hypothesis 
$$
\overline{\Pi}_{H}(x)\sim
\widehat{V}(\infty) \overline{\Pi}^+_{X}(x),\qquad x\to \infty.
$$ 
This  implies that for any $b,a>0$ and $0\leq y\leq 1$ such that $b-(1-y)>0$,
\begin{equation*}
\begin{split}
\frac{1}{\overline{\Pi}_{H}(x)}&\int_{[x(b-(1-y)),\infty)}\widehat{V}(dl) \overline{\Pi}^+_{X}(ax+l+x(1-y))\\
&=\frac{1}{\overline{\Pi}_{H}(x)}\int_{[x(b-(1-y)),\infty)}\widehat{V}(dl) \overline{\Pi}^+_{X}(l-x(b-(1-y))+x(a+b))\\
&\leq \frac{ \overline{\Pi}^+_{X}(x(a+b))}{\overline{\Pi}_{H}(x)}\int_{[x(b-(1-y)),\infty)}\widehat{V}(dl)\\
&\leq \frac{ \overline{\Pi}^+_{X}(x(a+b))}{\overline{\Pi}_{H}(x)}\int_{[xb,\infty)}\widehat{V}(dl)\xrightarrow[x\to\infty]{}0.
\end{split}
\end{equation*}
Therefore, using the above estimate and that $\overline{\Pi}_{H}(x)V[0,xc]\to c^{\alpha}/\alpha\Gamma(\alpha)\Gamma(1-\alpha),$ as $x\to\infty,$ uniformly for $c\in K,$ $K$ being any compact set in $(0,\infty)$ cf. \cite{bertoinsub}, in the identity (\ref{eq:OUregvar}), we obtain that
\begin{equation}\label{eq:OUregvar2}
\begin{split}
&\hspace{-2cm}\lim_{x\to\infty}\mathbb{P}\left(x-{X}_{\tau^+_{x}-}>xb, X_{\tau^+_{x}}-x>ax\right)\\
&=1_{\{0\leq b\leq 1\}}\lim_{x\to\infty}\int^{(1-b)}_{0}V(xdy)\int_{[0,\infty)}\widehat{V}(dl)\overline{\Pi}_X^+(ax+l+x(1-y)).
\end{split}
\end{equation}
Moreover, it follows from  (\ref{eqvigon}) that for $a> 0,$ $0\leq b\leq 1$ and $y\in[0,1-b],$ $$\overline{\Pi}_{H}(x(a+(1-y)))=\int_{[0,\infty)}\widehat{V}(dl)\overline{\Pi}_X^+(ax+l+x(1-y)),$$ and hence arguing as in the proof of Theorem~\ref{mainthmOUshoots}~(iii) we arrive at the identity
\begin{equation*}
\begin{split}
&\hspace{-2cm}\lim_{x\to\infty}\mathbb{P}\left(x-{X}_{\tau^+_{x}-}>xb, X_{\tau^+_{x}}-x>ax\right)\\
&=\lim_{x\to\infty}\int^{(1-b)}_{0}V(xdy)\int_{[0,\infty)}\widehat{V}(dl)\overline{\Pi}_X^+(ax+l+x(1-y))\\
&=\lim_{x\to\infty}\int^{(1-b)}_{0}\overline{\Pi}_{H}(x)V(xdy)\frac{\overline{\Pi}_{H}(x(a+(1-y)))}{\overline{\Pi}_{H}(x)}\\
&=\frac{1}{\Gamma(1-\alpha)\Gamma(\alpha)}\int^{(1-b)}_{0}y^{\alpha-1}\frac{1}{\left(a+(1-y)\right)^{\alpha}}dy 
\end{split}
\end{equation*}
Taking derivatives we obtain the form of the density of the asymptotic law for the overshoot and undershoot of $X$ claimed in Theorem~\ref{mainthmOUshoots}~(iv).   

The latter identity and Lemma \ref{lemmaOUshoot} allow us to conclude that for $0\leq u \leq v\leq 1,$ $w\geq0,$
\begin{equation*}
\begin{split}
\lim_{x\to\infty}\mathbb{P}&\left(\frac{x-\overline{X}_{\tau^+_{x}-}}{x}>u, \frac{x-{X}_{\tau^+_{x}-}}{x}>v, \frac{X_{\tau^+_{x}}-x}{x}>w \right)\\
&=\lim_{z\to\infty}\mathbb{P}\left(\frac{\mathcal{V}_{z}}{z}>\frac{v-u}{1-u}, \frac{\mathcal{O}_{z}}{z}>\frac{w+u}{1-u}\right)=\frac{1}{\Gamma(1-\alpha)\Gamma(\alpha)}\int^{\frac{1-v}{1-u}}_{0} y^{\alpha-1}\frac{1}{\left(\frac{1+w}{1-u}-y\right)^{\alpha}} dy\\
&=\frac{1}{\Gamma(1-\alpha)\Gamma(\alpha)}\int^{1-v}_{0} z^{\alpha-1}\frac{1}{\left(w+1-z\right)^{\alpha}}  dz=\frac{1}{\Gamma(1-\alpha)\Gamma(\alpha)}\int^\infty_{\frac{v+w}{1-v}}(1+y)^{-1}y^{-\alpha} dy.
\end{split}
\end{equation*}
The latter identity and the fact that $x-{X}_{\tau^+_{x}-}\geq x-\overline{X}_{\tau^+_{x}-},$ implies that in the present case the weak limit of $(x-{X}_{\tau^+_{x}-})/x$ equals that of $(x-\overline{X}_{\tau^+_{x}-})/x.$ To see this, note that  for $0\leq u< u+\epsilon\leq 1,$
\begin{equation*}
\begin{split}
&\hspace{-2cm}\lim_{x\to\infty}\mathbb{P}\left(\frac{x-\overline{X}_{\tau^+_{x}-}}{x}\in(u,u+\epsilon], \frac{x-{X}_{\tau^+_{x}-}}{x}>u+\epsilon\right)\\
&=\lim_{x\to\infty}\mathbb{P}\left(\frac{x-\overline{X}_{\tau^+_{x}-}}{x}>u, \frac{x-{X}_{\tau^+_{x}-}}{x}>u+\epsilon\right)\\
&\hspace{2cm}-\lim_{x\to\infty}\mathbb{P}\left(\frac{x-\overline{X}_{\tau^+_{x}-}}{x}>u+\epsilon, \frac{x-{X}_{\tau^+_{x}-}}{x}>u+\epsilon\right)\\
&=0.
\end{split}
\end{equation*}
This implies that for $0\leq u < u+\epsilon\leq 1,$
\begin{equation*}
\lim_{x\to\infty}\mathbb{P}\left(\frac{x-{X}_{\tau^+_{x}-}}{x}\in(u,u+\epsilon] \Big| \frac{x-\overline{X}_{\tau^+_{x}-}}{x}\in(u,u+\epsilon]\right)=1.
\end{equation*}
For $0<\epsilon<1,$ let $n_{\epsilon}$ be the largest integer such that $n_{\epsilon}\epsilon\leq 1.$ It follows that
\begin{equation*}
\begin{split}
\p&\left(\frac{x-{X}_{\tau^+_{x}-}}{x}-\frac{x-\overline{X}_{\tau^+_{x}-}}{x}>\epsilon\right)\\
&=\sum^{n_{\epsilon}}_{k=0}\p\left(\frac{x-{X}_{\tau^+_{x}-}}{x}-\frac{x-\overline{X}_{\tau^+_{x}-}}{x}>\epsilon \Big| \frac{x-\overline{X}_{\tau^+_{x}-}}{x}\in(k\epsilon,(k+1)\epsilon\wedge 1] \right)\\
&\hspace{8cm}\cdot\p\left(\frac{x-\overline{X}_{\tau^+_{x}-}}{x}\in(k\epsilon,(k+1)\epsilon\wedge 1]\right)\\
&\leq\sum^{n_{\epsilon}}_{k=0}\p\left(\frac{x-{X}_{\tau^+_{x}-}}{x}\notin(k\epsilon,(k+1)\epsilon\wedge 1] \Big| \frac{x-\overline{X}_{\tau^+_{x}-}}{x}\in(k\epsilon,(k+1)\epsilon\wedge 1] \right)\\
&\hspace{8cm}\cdot\p\left(\frac{x-\overline{X}_{\tau^+_{x}-}}{x}\in(k\epsilon,(k+1)\epsilon\wedge 1]\right) 
\end{split}
\end{equation*}
and that the right hand side tends to zero as $x\rightarrow\infty$.
\end{proof}
\begin{proof}[Proof of (i) in Theorem \ref{mainthmOUshoots:bis}]
The proof of this result uses similar arguments  to those used in the proof of (iii) in Theorem~\ref{mainthmOUshoots} so we will just provide the main steps of the proof. First  with the help of Lemma \ref{lemmaOUshoot}, we have  for $0\leq u<1,$ $v>-1,$ $w>0,$
\begin{equation*}
\begin{split}
&\mathbb{P}(x-\overline{X}_{\tau^+_{x}-}>ux, -X_{\tau^+_x-}>vx, X_{\tau^+_{x}}-x>wx,\tau^+_x<\infty)\\
&=\mathbb{P}(x(1-u) -X_{\tau^+_{x(1-u)}-}>(v+1-u)x, \, X_{\tau^+_{x(1-u)}}-x(1-u) >(w+u)x, \tau^+_{(1-u)x}<\infty)\\
&=\int^{x(1-u)}_{0}V(dy)\int_{[(vx+y)\vee 0,\infty)}\widehat{V}(dl) \overline{\Pi}^+_{X}((w+1+v)x+l-(vx+y))
\end{split}
\end{equation*}
Note that assumption (i) and the monotone density theorem for regularly varying functions imply that  $\int_x^\infty \overline{\Pi}^+_X(z)dz$ is regularly varying with index $-\alpha$ so that Theorem 3 in~\cite{Rsinai} is applicable. The latter theorem together with  Lemma 3.5 in \cite{KKM} and our  hypotheses imply that 
\begin{equation}\label{one}
\overline{\Pi}_{H}(x)\sim\frac{1}{\mu_{-}}\int^\infty_{x} \overline{\Pi}^+_{X}(z)dz\qquad x\to\infty,
\end{equation}
and 
\begin{equation}
\lim_{x\to\infty}\frac{\mathbb{P}(\tau^+_x<\infty)}{\overline{\Pi}_{H}(x)}=V(\infty).
\label{two}
\end{equation}
Note that in the above application of Theorem 3 in \cite{Rsinai} it is necessary to verity hypothesis (a-1). This boils down to checking that $$\lim_{x\to\infty}\frac{\int^{\infty}_{x+t}\overline{\Pi}^{+}_{X}(y)dy}{\int^{\infty}_{x}\overline{\Pi}^{+}_{X}(y)dy}=1,\qquad t\in\re.$$  However, this is a strightforward consequence of the fact that $\int_x^\infty \overline{\Pi}^+_X(z)dz$ is regularly varying at infinity.

We may now proceed to argue  that, thanks to the regular variation of $\overline{\Pi}_{H}$ and a dominated convergence argument, for $0\leq u<1,$ $v>0,$ $w>0,$
\begin{equation*}
\begin{split}
\mathbb{P}&(x-\overline{X}_{\tau^+_{x}-}>ux, -X_{\tau^+_x-}>vx, X_{\tau^+_{x}}-x>wx | \tau^+_x<\infty)\\
&\sim\left(\frac{\overline{\Pi}_{H}(x)}{\mathbb{P}(\tau^+_x<\infty)}\right)\left({\int^{x(1-u)}_{0}V(dy)}\right)\left(\frac{\frac{1}{\mu_{-}}\int^\infty_{x(1+w+v)} \overline{\Pi}^+_{X}(z)dz}{\overline{\Pi}_{H}(x)}\right)\sim(1+w+v)^{-\alpha}
\end{split}
\end{equation*}
as $x\rightarrow\infty$.
In principle, to complete the proof we are obliged to check convergence when $-1<v<0$. However, by noting that the established limiting triple law above is not a  defective distribution, the proof is complete.
\end{proof}
\begin{proof}[Proof of (ii) in Theorem \ref{mainthmOUshoots:bis}]
For the same reasons as in the proof of part  (i) of Theorem \ref{mainthmOUshoots:bis} we will just make a sketch of proof as follows. Thanks to Lemma~\ref{lemmaOUshoot} and the quintuple law in Theorem \ref{quintupleDK}
 it follows that for $0<u\leq 1,$ $w>0,$ $v\in\re,$ such that $v+w>0,$ 
\begin{equation}
\begin{split}
&\mathbb{P}\left(-\overline{X}_{\tau^+_x-}>-ux, -X_{\tau^+_x-}>a(x)v, \overline{X}_{\tau^+_x}-x>a(x)w, \tau^+_x<\infty\right)\\ 
&=\mathbb{P}\left(x-\overline{X}_{\tau^+_x-}>x(1-u), x-X_{\tau^+_x-}>a(x)v+x, \overline{X}_{\tau^+_x}-x>a(x)w, \tau^+_x<\infty\right)\\
&=\mathbb{P}\left( ux-X_{\tau^+_{ux}-}   >va(x)+xu,  \, \overline{X}_{\tau^+_{xu}}-xu   >a(x)w+x(1-u), \tau^+_{xu}<\infty\right)\\
&=\int^{xu}_{0}V(dy)\int_{[(va(x)+y) \vee 0,\infty)}\widehat{V}(dl) \overline{\Pi}^+_{X}(wa(x)+x(1-u)+l+xu-y)\\
&=\int^{xu}_{0}V(dy)\int_{[(va(x)+y) \vee 0,\infty)}\widehat{V}(dl) \overline{\Pi}^+_{X}\left((w+v)a(x)+x+l-(a(x)v+y)\right),
\end{split}
\end{equation}
We recall that by Theorem 3 in~\cite{Rsinai} the assumption that $\Pi$ is subexponential implies that $\Pi_{H}$ is long-tailed $$\lim_{x\to\infty}\frac{\overline{\Pi}_{H}(x+t)}{\overline{\Pi}_{H}(x)}=1,\qquad \text{for each}\ t\in \re,$$ and moreover  that  (\ref{one}) and (\ref{two}) hold.
Arguing as in the proof of (iii) in Theorem \ref{mainthmOUshoots}, when $va(x)+y>0$, we have that 
\begin{equation*}
\begin{split}
&\int_{[(va(x)+y)\vee 0,\infty)}\widehat{V}(dl) \overline{\Pi}^+_{X}((w+v)a(x)+x+l-(a(x)v+y))\\
&\sim \frac{1}{\mu_{-}}\int^\infty_{x+a(x)(v+w)}dl \overline{\Pi}^+_{X}(l),\qquad x\to\infty.
\end{split}
\end{equation*}
In the case where $a(x)v+y<0,$ Vigon's identity (\ref{eqvigon}) and long-tailed behaviour imply that 
\begin{equation*}
\begin{split}
&\hspace{-2cm}\int_{[(va(x)+y) \vee 0,\infty)}\widehat{V}(dl) \overline{\Pi}^+_{X}\left((w+v)a(x)+x+l-(a(x)v+y)\right)\\
&\hspace{2cm}=\overline{\Pi}_{H}\left(wa(x)+x-y\right)\sim \overline{\Pi}_{H}\left(wa(x)+x\right)
\end{split}
\end{equation*}
for each $y$ as $x\rightarrow\infty$.

Putting the pieces together it follows that for $0<u\leq 1,$ $v,w>0$
\begin{equation*}
\begin{split}
&\mathbb{P}\left(-\overline{X}_{\tau^+_x-}>-ux, -X_{\tau^+_x-}>a(x)v, \overline{X}_{\tau^+_x}-x>a(x)w | \tau^+_x<\infty\right)\\
&\sim\frac{\overline{\Pi}_{H}(x)}{\mathbb{P}(\tau^+_x<\infty)}  V(\infty) \frac{\frac{1}{\mu_{-}}\int^\infty_{a(x)(v+w)+x} \overline{\Pi}^+_{X}(z)dz}{\frac{1}{\mu_{-}}\int^\infty_{x} \overline{\Pi}^+_{X}(z)dz}\sim e^{-(v+w)}.
\end{split} 
\end{equation*}
where the final estimate is obtained by L'H\^opital's rule using the fact that $a$ is differentiable and $a'(x)\to 0,$ as $x\to \infty.$ 

Once again it is not necessary to consider the case that $v<0$ as the triple law established above is not defective.
Note in particular that in this setting the weak limit of $\overline{X}_{\tau^+_x-}/x,$ conditionally on $\tau^+_x<\infty,$ is $0$ as $x\to\infty.$
\end{proof}
\section{Triple and quadruple laws at first and last passage of positive self-similar Markov processes}\label{triplePSSMP}

The objective in this section is to bring some of the results from Sections \ref{5-7} and \ref{asymptotics} into the setting of positive self-similar Markov processes. Although this will be a relatively straightforward operation, it will allow us to construct many new explicit examples in the following section.

A positive  Markov process $Y=(Y_{t}, t\geq 0)$ with
c\`adl\`ag paths is a self-similar process if for every $k>0$ and
every initial state $x\geq 0$ it satisfies the scaling property,
i.e., for some $\alpha>0$
\begin{equation*}
\textrm{the law of } (kY_{k^{-\alpha}t},t\geq 0) \textrm{ under
}\mathbf{P}_{x}\textrm{ is } \mathbf{P}_{kx},
\end{equation*}
where $\mathbf{P}_{x}$ denotes the law of the process $Y$ starting from
$x\geq 0$. Here, we use the notation $Y^{(x)}$ or $(Y, \mathbf{P}_x)$ for a positive self-similar Markov process starting from $x\ge 0$. Well-known examples of such class of processes include Bessel processes, stable subordinators or more generally  stable processes conditioned to stay positive.

Lamperti \cite{La} proved that there is a bijective correspondence between the class of  positive self-similar Markov processes  that never hit the state $0$ and the class of L\'evy processes which do not drift to $-\infty$. More precisely, let $Y^{(x)}$ be a self-similar
Markov process started from $x>0$ that fulfills the scaling property
for some $\alpha>0$, then
\begin{equation}\label{lamp}
Y^{(x)}_{t}=x\exp\Big\{X_{\theta(tx^{-\alpha})}\Big\},\qquad 0\leq
t\leq x^{\alpha}I(X),
\end{equation}
where
\begin{equation*}
\theta_{t}=\inf\Big\{s\geq 0: I_{s}(X)>t\Big\},\qquad
I_{s}(X)=\int_{0}^{s}\exp\big\{\alpha X_{u}\big\}d u,\qquad
I(X)=\lim_{t\to+\infty}I_{t}(X),
\end{equation*}
and $X$ is a L\'evy process starting from $0$ which does not drift to $-\infty$ and  whose law does not depend on $x>0$, here denoted by $\p$. This is the so-called Lamperti representation. 

Recall that $H$ denotes the ascending ladder height process associated to $X$ and $V$ its corresponding bivariate renewal function. Similarly,   $ \widehat{V}$ denotes the bivariate renewal function associated to the descending ladder processes  and  $\Pi_X$  the L\'evy measure of $X$.
 
Caballero and Chaumont \cite{CC} studied  the problem of when an entrance law at $0$ for $(Y, \mathbf{P}_x)$ can be defined. In particular, the authors in \cite{CC} gave necessary and sufficient
conditions for the weak convergence of $Y^{(x)}$ on the Skorokhod's
space, as $x$ goes to $0$, towards a non degenerate process, that we will denote by $Y^{(0)}$ on some occasions and $(Y,\mathbf{P}_0)$ on others. The limit process $Y^{(0)}$ is a positive self-similar Markov process which starts from $0$ continuously, it  fulfills the Feller property on $[0,\infty)$
and possesses  the same transition functions as $Y^{(x)}$, $x>0$.

According to Caballero and Chaumont \cite{CC}, necessary and
sufficient conditions for the weak convergence of $Y^{(x)}$ on the
Skorokhod's space are: $X$ is not arithmetic, $ \mu_+:=\mathbb{E}(H_1)<\infty$  and
\begin{equation}\label{condex}
\mathbb{E}\left(\log^+ \int_0^{\tau^+_1}\exp\big\{\alpha X_s\big\} d
s\right)<\infty,
\end{equation}
where $\tau^+_x$ is the first passage time above $x\geq 0$. Recently,  Chaumont et al. \cite{CKPR} proved that the additional hypothesis (\ref{condex}) is always satisfied whenever the L\'evy process $X$ is not arithmetic and $ \mu_+<\infty$.

For $b>x\geq 0$, we set
\[
T^{(x)}_{b}=\inf\{t\geq 0: Y^{(x)}_{t}\geq b\} \qquad \textrm{and}\qquad M^{(x)}_t=\sup_{0\leq s\leq t}Y^{(x)}_s.
\]
The first result of this section consist of computing the law of the triplet $(M^{(x)}_{T^{(x)}_b-},Y^{(x)}_{T^{(x)}_b-}, Y^{(x)}_{T^{(x)}_b} )$ and may be considered as a corollary to Theorem \ref{quintupleDK}. Recall that we drop the dependency on $x$ in the aforementioned random variable when the point of issue is indicated in the measure.
\begin{corollary} \label{cor-pmasp}For $0<x<b$, we have on $u\in[x,b),\, 0<v\leq u$ and $w> b$  
\begin{equation}\label{eq:firstentrance}
\begin{split}
&\mathbf{P}_{x}\left(M_{T_{b}-}\in d u, Y_{T_{b}-} \in dv, Y_{T_{b}}\in d w \right)\\
&=V(\log (b/x) +  du/u)\widehat{V}(\log (u/b)-dv/v)\Pi_X(d w/w-\log (v/b)).
\end{split}
\end{equation} 
where the equality holds up to a multiplicative constant.

Moreover, if $X$ is not arithmetic and $\mu_{+}<\infty,$ we have on  $0< v\leq u<b$ and $w>b, $
\begin{equation}\label{eq:firstentranceasym}
\begin{split}
\mathbf{P}_{0}\left(M_{T_{b}-}<u, Y_{T_{b}-}<v, Y_{T_{b}}> w\right)&=\frac{1}{\mu_{+}}\int_{0}^{\log (u/v)} dy \int_{[y,\infty)}\widehat{V}(d l)\overline{\Pi}_X^+(\log (w/v)+l-y)\\
&+\frac{1}{\mu_{+}}\int_{\log (b/v)} ^{\infty} d y \overline{\Pi}_H(\log (w/b)+y).
\end{split}
\end{equation}
\end{corollary}
\begin{proof}
Suppose that $F:\R_+^3\to \R_+$ is a measurable and bounded function such that $F(\cdot, \cdot, b)=0$.  From the Lamperti representation, it is clear that  
\[
\begin{split}
\mathbf{E}_x\Big(F(M_{T_{b}-}, Y_{T_{b}-} , Y_{T_{b}} )\Big)&=\mathbb{E}\Big(F\big(x\exp\{\overline{X}_{\tau^+_{\log b/x}-}\}, x\exp\{X_{\tau^+_{\log b/x}-}\}, x\exp\{X_{\tau^+_{\log b/x}}\}\big)\Big)\\
&=\mathbb{E}\Big(F\big(b e^{\overline{X}_{\tau^+_{\log b/x}-}-\log (b/x)}, be^{X_{\tau^+_{\log b/x}-}-\log (b/x)}, b e^{X_{\tau^+_{\log b/x}}-\log (b/x)}\big)\Big).
\end{split}
\]
Therefore, the identity (\ref{eq:firstentrance}) follows using Theorem 1 and 
straightforward computations.

Now, let $F:\R_+^3\to \R_+$ be a bounded and  continuous function such that $F(\cdot, \cdot, b)=0$ and suppose that $X$ is not arithmetic and $\mu_+<\infty$. From the scaling property, we have  that for every  positive real constant $c$,
\[
\mathbf{E}_0\Big(F(M_{T_{b}-}, Y_{T_{b}-} , Y_{T_{b}} )\Big)=\mathbf{E}_0\Big(F(c^{-1}M_{T_{cb}-},c^{-1} Y_{T_{cb}-} , c^{-1}Y_{T_{cb}} )\Big).
\]
Let $0<\epsilon<c\,b$, hence 
\begin{equation}\label{cortando}
\begin{split}
\mathbf{E}_0\Big(F(c^{-1}M_{T_{cb}-},c^{-1} Y_{T_{cb}-} , c^{-1}Y_{T_{cb}} )\Big)&=\mathbf{E}_0\Big(F(c^{-1}M_{T_{cb}-},c^{-1} Y_{T_{cb}-} , c^{-1}Y_{T_{cb}} ), T_\epsilon<T_{bc}\Big)\\
&+\mathbf{E}_0\Big(F(c^{-1}M_{T_{cb}-},c^{-1} Y_{T_{cb}-} , c^{-1}Y_{T_{cb}} ), T_\epsilon=T_{cb}\Big).
\end{split}
\end{equation}
From the Markov property and the Lamperti representation, the first term of the right hand-side of (\ref{cortando}) satisfies that
\begin{equation}\label{mplrlim}
\begin{split}
\mathbf{E}_0&\Big(F(c^{-1}M_{T_{cb}-},c^{-1} Y_{T_{cb}-} , c^{-1}Y_{T_{cb}} ), T_\epsilon<T_{bc}\Big)\\
&=\int_{\epsilon}^{cb}\mathbf{P}_0(Y_{T_\epsilon}\in d x) \mathbf{E}_x\Big(F(c^{-1}M_{T_{cb}-},c^{-1} Y_{T_{cb}-} , c^{-1}Y_{T_{cb}} )\Big)\\
&=\int_{\epsilon}^{cb}\mathbf{P}_0(Y_{T_\epsilon}\in d x)\mathbb{E}\Big(F\big(b e^{\overline{X}_{\tau^+_{\log cb/x}-}-\log (cb/x)}, be^{X_{\tau^+_{\log cb/x}-}-\log (cb/x)}, b e^{X_{\tau^+_{\log cb/x}}-\log (cb/x)}\big)\Big).
\end{split}
\end{equation}
On the other hand, since $F$ is  bounded by some positive constant, say $k>0$, and the scaling property, we get 
\[
0\le \mathbf{E}_0\Big(F(c^{-1}M_{T_{cb}-},c^{-1} Y_{T_{cb}-} , c^{-1}Y_{T_{cb}} ), T_\epsilon=T_{cb}\Big)\le k\, \mathbf{P}_0( T_{c^{-1}\epsilon}=T_{b}).
\]
Hence, the second term  of the right hand-side of (\ref{cortando}) goes to $0$, as $c$ tends to $\infty$, since $\lim_{c\to \infty}\mathbf{P}_0( T_{c^{-1}\epsilon}=T_{b})=0$. From identity (\ref{mplrlim}), Theorem 3 part (i) and the dominated convergence Theorem, we deduce  
\[
\begin{split}
\lim_{c\to\infty}\mathbf{E}_0&\Big(F(c^{-1}M_{T_{cb}-},c^{-1} Y_{T_{cb}-} , c^{-1}Y_{T_{cb}} ), T_\epsilon<T_{bc}\Big)\\
&=\lim_{c\to\infty}\mathbb{E}\Big(F\big(b e^{\overline{X}_{\tau^+_{\log cb/x}-}-\log (cb/x)}, be^{X_{\tau^+_{\log cb/x}-}-\log (cb/x)}, b e^{X_{\tau^+_{\log cb/x}}-\log (cb/x)}\big)\Big). 
\end{split}
\]
Putting the pieces together, we conclude that 
\[
\mathbf{E}_0\Big(F(M_{T_{b}-}, Y_{T_{b}-} , Y_{T_{b}} )\Big)=\lim_{y\to\infty}\mathbb{E}\Big(F\big(b e^{\overline{X}_{\tau^+_{\log y}-}-\log y}, be^{X_{\tau^+_{\log y}-}-\log y}, b e^{X_{\tau^+_{\log y}}-\log y}\big)\Big).
\]
The identity (\ref{eq:firstentranceasym}) follows from  (\ref{oushoots}) after some straightforward computations.
\end{proof}
Let $x\ge 0$ and take $ b>0$. We set 
\[
\sigma^{(x)}_b=\sup\Big\{s\ge 0: Y^{(x)}_s\le b\Big\}\qquad\textrm{ and }\qquad
J^{(x)}_t=\inf_{s\ge t} Y^{(x)}_s.
\] 
The next result deals with the computation of the law of $(1/J^{(x)}_0, Y^{(x)}_{\sigma^{(x)}_b-},  Y^{(x)}_{\sigma^{(x)}_b},  J^{(x)}_{\sigma^{(x)}_b} )$, for $x>0$. 
\begin{corollary} \label{log-trans}Suppose that the underlying L\'evy process $X$ is regular for both $(0,\infty)$ and $(-\infty, 0)$. For $x,b>0$, we have on $v\geq x^{-1}\lor b^{-1},  v^{-1}<y<b$ and $ b<u\leq w<\infty$   
\begin{equation}\label{eq:firstentrance1}
\begin{split}
&\mathbf{P}_{x}\Big( 1/J_0\in dv,\, Y_{\sigma_b-}\in dy, \, Y_{\sigma_b}\in dw,\,  J_{\sigma_b}\in du\Big)\\
&=\widehat{V}(dv/v)V(\log (bv) +  dy/y)\Pi_X(d w/w-\log(y/b))\widehat{V}(\log (w/b)-du/u)
\end{split}
\end{equation} 
where the equality holds up to a multiplicative constant.
\end{corollary}
\begin{proof} We first suppose that  $0<b \leq x$ and take $F:\R_+^4\to \R_+$  a measurable and bounded function such that $F(\cdot,\cdot, \cdot, b)=0$.  From the Lamperti representation, it is clear that  
\[
\begin{split}
&\mathbf{E}_x\Big(F(1/J_0,\, Y_{\sigma_b-}, \, Y_{\sigma_b},\,  J_{\sigma_b})1_{\{J_0<b\}}\Big)=\\
&=\mathbb{E}\Big(F\big(x^{-1}e^{-\underline{X}_\infty}, x e^{X_{U_{\log b/x}-}}, x e^{X_{U_{\log b/x}}}, xe^{\underrightarrow{X}_{U_{\log b/x}}}\big)1_{\{\underline{X}_\infty<\log b/x\}}\Big)\\
&=\mathbb{E}\Big(F\big(x^{-1}e^{-\underline{X}_\infty},xe^{\underline{X}_\infty}e^{\tilde{X}_{\tilde{U}_{\log b/x-\underline{X}_\infty}-}}, xe^{\underline{X}_\infty} e^{\tilde{X}_{\tilde{U}_{\log b/x-\underline{X}_\infty}}}, xe^{\underline{X}_\infty}e^{\underrightarrow{\tilde{X}}_{\tilde{U}_{\log b/x-\underline{X}_\infty}}}\big)1_{\{\underline{X}_\infty<\log b/x\}}\Big),
\end{split}
\]
where $\tilde{X}=(X_{G_{\infty}+t}-X_{G_{\infty}}, t\ge 0)$, $G_{\infty}$ is the time when the process $X$ reaches is global minimum and $\tilde{U}_x$ denotes the last passage time of $\tilde{X}$ above the level $x$. In particular, note that when $0$ is regular for $(-\infty,0)$ and for $(0,\infty)$ the following identity holds $X_{G_{\infty}}=\underline{X}_\infty$. 

On the other hand, from Millar's path decomposition (see Theorem 3.1 in \cite{Mi}) we deduce that 
 \[
\begin{split}
\mathbb{E}&\Big(F\big(x^{-1}e^{-\underline{X}_\infty},xe^{\underline{X}_\infty}e^{\tilde{X}_{\tilde{U}_{\log b/x-\underline{X}_\infty}-}}, xe^{\underline{X}_\infty} e^{\tilde{X}_{\tilde{U}_{\log b/x-\underline{X}_\infty}}}, xe^{\underline{X}_\infty}e^{\underrightarrow{\tilde{X}}_{\tilde{U}_{\log b/x-\underline{X}_\infty}}}\big)1_{\{\underline{X}_\infty<\log b/x\}}\Big)\\
&=\int_{\log x/b}^{\infty}\widehat{V}( dv)\mathbb{E}^{\uparrow}\Big(F\big(x^{-1}e^{v},xe^{-v}e^{X_{U_{\log b/x+v}-}}, xe^{-v} e^{X_{U_{\log b/x+v}}}, xe^{-v}e^{\underrightarrow{X}_{U_{\log b/x+v}}}\big) \Big).
\end{split}
\]
Then our assertion follows  from Theorem 2 and straightforward computations.

The case when $0<x<b$ is much simpler, since we do not need to decompose the process at its global infimum. We may proceed as above,  using Lamperti representation and then Corollary 2 to get (\ref{eq:firstentrance1}).
\end{proof}
Finally, we are interested in computing the law of $(Y^{(0)}_{\sigma^{(0)}_b -}, Y^{(0)}_{\sigma^{(0)}_b}, J^{(0)}_{\sigma^{(0)}_b})$. The distribution of  $Y^{(0)}_{\sigma^{(0)}_b -}$ has been recently characterized in \cite{CP}, as follows: $b^{-1}Y^{(0)}_{\sigma^{(0)}_b -}\ed e^{-\mathfrak{U}Z},$ where $\mathfrak{U}$ and $Z$ are independent random variables, $\mathfrak{U}$ is uniformly distributed over $[0,1]$ and the law of $Z$ is given by 
\[
\mathbb{P}(Z>u)=\frac{1}{\mu_+}\int_{u}^{\infty}s\Pi_{H}(ds), \qquad u\ge 0.
\]
Before we state the last result of this section, let us recall a path decomposition, introduced in \cite{CP}, of the  positive self-similar Markov process $(Y, \mathbf{P}_0)$ time-reversed at  last passage  which is associated to the L\'evy process $X$. Fix a decreasing sequence $(x_n, n\geq 0)$ of positive real numbers which tends to $0$. From Corollary 1 in \cite{CP}, we have
\[
\Big(Y^{(0)}_{(\sigma^{(0)}_{x_n}-t)-},\, 0\le t\leq \sigma^{(0)}_{x_n}-\sigma^{(0)}_{x_{n+1}}\Big)=\left(\Gamma_{n}\exp\Big\{X^{n}_{\theta^{n}(t/\Gamma_n)}\Big\},\, 0\le t\le H_n\right), \quad n\geq 0,
\]
where the processes $X^n, n\geq 0$ are mutually independent and have the same law as $\hat{X}=-X$. Moreover the sequence $(X^n, n\geq 0)$ is independent of $Y_{\sigma^{(0)}_{x_0}-}:=\Gamma_0$ and
\[\left\{\begin{array}{l}
\theta^n(t)=\inf\left\{s:\displaystyle\int_0^s\exp\Big\{X_u^{n}\Big\}\,du\ge t\right\}\\
H_n=\Gamma_n\displaystyle\int_0^{\tau^{n}(\log(x_{n+1}/\Gamma_n))}\exp\Big\{X^{n}_s\Big\}\,ds\\
\Gamma_{n+1}=\Gamma_{n}\exp\Big\{X^{n}_{\tau^{n}(\log
(x_{n+1}/\Gamma_{n}))}\Big\},\,n\ge 0, \\
\tau^n(z)=\inf\{t:X^{n}_t\le z\}.
\end{array}\right.\]
Moreover for each $n$, $\Gamma_n$ is independent of $X^n$ and 
\begin{equation}\label{identym}
x_{n}^{-1}\Gamma_n\ed x_{1}\Gamma_1.
\end{equation}
\begin{proposition}
Let $b>0$ and assume that $X$ is not arithmetic and $\mu_+<\infty$. Then the following identity holds 
\[
\begin{split}
\mathbb{P}_0\Big(Y_{\sigma_b^-}<\omega, Y_{\sigma_b}>v, J_{\sigma_b}>u\Big)&=\frac{1}{\mu_+}\int_0^{\log(v/u)} dx\int_{[x,\infty)}\widehat{V}(dl)\overline{\Pi}_X^+(\log(v/\omega)+l-x)\\
&+\frac{1}{\mu_+}\int_{\log(v/b)}^{\infty}dx\, \overline{\Pi}_H(\log(b/\omega)+x),
\end{split}
\]
where $0<\omega\leq b \leq  u\leq v$.
\end{proposition}


\begin{proof} Take a decreasing sequence $(x_n, n\ge 0)$ of positive real numbers converging to $0$ and such that $x_0>b$ and $x_1=b$. By the scaling property, it is clear that for each $n\geq 1$,
\begin{equation}\label{3pmasp}
x^{-1}_{n}J^{(0)}_{\sigma^{(0)}_{x_n}}\ed x_1^{-1} J^{(0)}_{\sigma^{(0)}_{x_1}}, \qquad x^{-1}_{n}Y^{(0)}_{\sigma^{(0)}_{x_n}}\ed x_1^{-1} Y^{(0)}_{\sigma^{(0)}_{x_1}}\quad \textrm{and} \quad x^{-1}_{n}Y^{(0)}_{\sigma^{(0)}_{x_n}-}\ed x_1^{-1} Y^{(0)}_{\sigma^{(0)}_{x_1}-}.
\end{equation}
Now using the path decomposition of the process $Y^{(0)}$ reversed at $\sigma^{(0)}_{x_0}$ described above, we deduce that the first identity in law in (\ref{3pmasp}) can be written as follows
\[
x_n^{-1}\Gamma_{n-1}e^{\underline{X}^{n-1}_{\tau^{n-1}({\log(x_{n}/\Gamma_{n-1}))-}}}\ed x_1^{-1} J^{(0)}_{\sigma^{(0)}_{x_1}},
\]
where $\underline{X}^{n-1}_t=\inf_{0\leq u\le t}X^{n-1}_u$. Similarly, we have that
$x_n^{-1}\Gamma_{n-1}e^{X^{n-1}_{\tau^{n-1}(\log(x_{n}/\Gamma_{n-1}))-}}\ed x_1^{-1} Y^{(0)}_{\sigma^{(0)}_{x_1}} $ and $x_n^{-1}\Gamma_{n}\ed x_1^{-1} Y^{(0)}_{\sigma^{(0)}_{x_1}-}.$
Recall that $ x_1^{-1} Y^{(0)}_{\sigma^{(0)}_{x_1}-}=e^{-\mathfrak{U}Z }$. By the independence of $X^{n-1}$ and $\Gamma_{n-1}$ and the identity (\ref{identym}), we deduce
\[
x_1^{-1} J^{(0)}_{\sigma^{(0)}_{x_1}}\ed e^{\underline{X}^{n-1}_{\tau^{n-1}({\log(x_{n}/\Gamma_{n-1}))-}}-\log(x_{n}/\Gamma_{n-1})}\ed e^{\underline{\widehat{X}}_{\widehat{\tau}(\log(x_{n}/x_{n-1})-\mathfrak{U}Z)^-}-\log(x_{n}/x_{n-1})+\mathfrak{U}Z }.
\]
From the same arguments as above, we have that 
\[
x_1^{-1} Y^{(0)}_{\sigma^{(0)}_{x_1}} \ed e^{X^{n-1}_{\tau^{n-1}(\log(x_{n}/\Gamma_{n-1}))^-}-\log(x_{n}/\Gamma_{n-1})}
\ed e^{\widehat{X}_{\widehat{\tau}(\log(x_{n}/x_{n-1})-\mathfrak{U}Z)-}-\log(x_{n}/x_{n-1})+\mathfrak{U}Z },
\]
Then by taking $x_n=be^{-n^2}$ for $n\geq 2$, we deduce from the above equalities that $\log(x_1^{-1} J^{(0)}_{\sigma^{(0)}_{x_1}})$ and $\log(x_1^{-1} X^{(0)}_{\sigma^{(0)}_{x_1}} )$ have the same limit as the limit undershoot of the processes $(\underline{\widehat{X}}_t , t\ge 0)$ and $X$, respectively, i.e.
\[
\underline{\widehat{X}}_{\widehat{\tau}(x)-}-x\to \log(x_1^{-1} J^{(0)}_{\sigma^{(0)}_{x_1}})\quad\textrm{ and }\quad x-\widehat{X}_{\widehat{\tau}(x)-}\to   \log(x_1^{-1} Y^{(0)}_{\sigma^{(0)}_{x_1}} ),
\] 
in law as $x$ tends to $-\infty$. Hence,  Theorem 3 part $(i)$ gives us the desired result.
\end{proof}
\section{Some explicit examples.}\label{egs}
We conclude our exposition by offering a number of explicit examples. 
A significant number of these examples are the result of inter-playing
the role of a stable process until it first exits $(0,\infty)$,  and conditioned versions thereof,  as both are self-similar Markov process as well as (Doob $h$-transforms of) a L\'evy process. In this respect, the routine calculations in the previous section will prove to have been very useful. Note, it is straightforward to check that all the L\'evy processes mentioned below are regular for both $(0,\infty)$ and $(-\infty, 0)$.

\subsection{Conditioned stable processes and last passage times}

 Suppose that $X$ is a stable L\'evy process with index $\alpha\in (0,2)$, i.e. a L\'evy process satisfying the scaling property with index $\alpha$. It is known that its  L\'evy measure is given by
\[
\Pi_X( dx )=1_{\{x>0\}}\frac{c_+}{x^{1+\alpha}}dx +1_{\{x<0\}}\frac{c_-}{x^{1+\alpha}}dx,
\]  
where $c_+$ and $c_{-}$ are two nonnegative real numbers (see for instance \cite{be1}). To avoid  trivialities, we assume $c_+>0$.

It is known (cf. Bertoin \cite{be1}) that the ladder process $H$ of a stable process of index $\alpha$ is a stable subordinator with index $\alpha\rho$, where $\rho=\p(X_1\geq 0)$ (positivity parameter) and, hence, up to a multiplicative constant $\kappa(0,\beta)=\beta^{\alpha\rho}$ for $\beta\geq 0$. In a similar way,  up to a multiplicative constant $\widehat{\kappa}(0,\beta)=\beta^{\alpha(1-\rho)}$. From (\ref{doubleLT}) it can easily be shown  that (up to a multiplicative constant) 
\begin{equation*}
V(d x)=\frac{x^{\alpha\rho-1}}{\Gamma(\alpha\rho)}d x \qquad\textrm{and}\qquad \widehat{V}(d x)=\frac{x^{\alpha(1-\rho)-1}}{\Gamma(\alpha(1-\rho))}d x. 
\end{equation*}
The form of the law of the triple law of undershoots and overshoot for a stable process can be read from (\ref{eq:SOU3L})

Marginalizing the quintuple law for the stable process conditioned to stay positive (see Theorem 2), we now obtain the following new  identity.

\begin{corollary}For $0<y\leq x$ and $0<u\le w$,
\[
\begin{split}
&\hspace{-1cm}\p^{\uparrow}(\underrightarrow{X}_{U_x}-x\in du, x-X_{U_x-}\in dy, X_{U_x}-x\in d w)\\
&=\frac{\sin (\pi\alpha\rho)}{\pi}\,\frac{\Gamma(\alpha+1)}{\Gamma(\alpha\rho)\Gamma(\alpha(1-\rho))}
\frac{(x-y)^{\alpha\rho-1}(w-u)^{\alpha(1-\rho)-1}}{(w+y)^{\alpha+1}}du\, dy\,dw.
\end{split}
\] 
 
\end{corollary}


\noindent Note that the  normalizing constant above is chosen to make the density on the  right-hand side a distribution. In particular,  stable processes do not creep and hence from Tanaka's path construction, we deduce that it is not necessary to take care of  an atom on the event $\{X_{U_x}=x\}$.

When the stable process conditioned to stay positive starts from $z>0$, Corollary 1 give us the following  identity.

\begin{corollary}
For $x>0$, $0< v< z\land x$, $0<y\leq  x-v$ and $0<u\leq w$,
\[
\begin{split}
\p^{\uparrow}_z(\underline{X}_{\infty}\in dv,& \underrightarrow{X}_{U_x}-x\in du, x-X_{U_x-}\in dy, X_{U_x}-x\in d w)\\
&=K_1(x,z)\frac{(z-v)^{\alpha(1-\rho)-1}(x-v-y)^{\alpha\rho-1}(w-u)^{\alpha(1-\rho)-1}}{z^{\alpha(1-\rho)}(w+y)^{\alpha+1}} dv \,du\, dy\,dw.
\end{split}
\] 
The normalizing constant $K_1(x,z)$ (which depends on $x$ and $z$) makes the right-hand side of the previous identity a distribution and following a  quadruple integral can be shown to be
\[
K_1(x,z)=\frac{sin (\pi\alpha\rho)}{\pi}\,\frac{\alpha (1-\rho)\Gamma(\alpha +1)}{\Gamma(\alpha\rho)\Gamma(\alpha(1-\rho))}\bigg(1-\Big(1-\frac{z\land x}{z}\Big)^{\alpha(1-\rho)}\bigg)^{-1}.
\]  


\end{corollary}


\subsection{Lamperti-stable processes: I}\label{lamperti-I}

A particular family of L\'evy processes which will be of interest to us in this and subsequent examples are Lamperti-stable  process with  characteristics  $(\varrho, \beta, \gamma)$ where $\varrho\in (0,2)$  and $\beta, \gamma\le \varrho +1$. Such L\'evy process have no Gaussian component and their L\'evy measure is of the type
\begin{equation}\label{mlls}
1_{\{x>0\}}\frac{c_+e^{\beta x}}{(e^x-1)^{1+\varrho}}dx +1_{\{x<0\}}\frac{c_- e^{-\gamma x}}{(e^{-x}-1)^{1+\varrho}}dx,
\end{equation}
where $c_+$ and $c_{-}$ are two nonnegative real numbers. We refer to  \cite{ CPP} for a proper definition and \cite{CC1,CKP} for more details in what follows. We also mention the work of \cite{kut} in which a larger class of L\'evy processes (called the $\beta$-class)   is defined. We shall predominantly be concerned with the case that $c_+>0$.
In the forthcoming text we shall also make reference to  Lamperti-stable subordinators with characteristics $(\varrho, \gamma)$. In that case we mean a (possibly killed) subordinator which has no drift term and L\'evy measure 
of the form
\begin{equation}
1_{\{x>0\}}c\frac{e^{\gamma x}}{(e^x-1)^{1+\varrho}}dx
\label{LS-sub}
\end{equation}
for $c>0$, $\gamma\leq 1+\varrho$ and $\varrho\in(0,1)$.


Lamperti-stable processes  occour naturally when considering an $\alpha$-stable process conditioned to stay positive. 
Indeed, the latter processes are self-similar and never hit the origin and hence respect the Lamperti representation (\ref{lamp}). More formally (keeping with the same notation as in the previous sub-section)
when $X$ is issued from $x>0$ we may write
\begin{equation}\label{lamp=}
X_t=x\exp\left\{\xi^\uparrow_{\theta(tx^{-\alpha})}\right\} 
\end{equation}
where for $t>0$,
\[\theta(t) = \inf \left\{s\geq 0 : \int_0^s
\exp\left\{\alpha\xi^\uparrow_u\right\} {\rm d} u \geq t\right\}.
\]
Moreover, the process
 $\xi^\uparrow=(\xi^\uparrow_t,\;t\geq 0)$ is a L\'evy process started from $0$
whose law does not depend on $x>0$ and 
in Caballero and Chaumont \cite{CC1} it was shown that $\xi^\uparrow$ is a Lamperti-stable process with characteristics given by   $\varrho=\alpha$, $\beta=\alpha(1-\rho) +1$ and $\gamma=\alpha\rho$, where $\rho$ is the positivity parameter of the associated stable process.

Our objective in this section is to offer some explicit identities for the process $\xi^\uparrow$. In that case the first and last passage times, ie. $\tau^+_\cdot$ and $U_\cdot$, as well the notation for the  running maximum and minimum should be understood accordingly.
We  recall  that  the stable process conditioned to stay positive drifts to $+\infty$,  from the Lamperti representation (\ref{lamp=}) we deduce that the process $\xi^\uparrow$ also drifts to $+\infty$.
The law of the overall infimum of $\xi^\uparrow$ has been computed in Proposition 2  of \cite{CC1} (see also Corollary 2 in \cite{CKP}) which is given by
\[
\p(-\underline{\xi}^\uparrow_{\infty}\le z)=(1-e^{-z})^{\alpha(1-\rho)},\qquad \textrm{ for all } \quad z\geq 0,
\]
which implies, by Proposition VI.17 in \cite{be1},  that the renewal function $\widehat{V}$ can be represented as follows
\begin{equation}\label{desrfls}
\widehat{V}(z)=\widehat{V}(\infty)
(1-e^{-z})^{\alpha(1-\rho)},\qquad \textrm{ for all } \quad z\geq 0.
\end{equation}
It is well known that $\widehat{V}$ is unique up to a multiplicative constant which depends on the normalization of local time of $\xi^\uparrow$ at its infimum. Without loss of generality we may therefore assume in the forthcoming analysis that $\widehat{V}(\infty)$, which is equal to the reciprocal of killing rate of the descending ladder height process, may be taken identically equal to 1.
In this respect we shall also assume that $c_+=1$.

With these assumptions in place, we  find by (\ref{doubleLT}) that
\[
\widehat{\kappa}(0,\lambda)=\frac{1}{
\Gamma(\alpha(1-\rho)+1)}\frac{\Gamma(\alpha(1-\rho)+1+\lambda)}{\Gamma(\lambda+1)},\qquad \textrm{ for all } \quad \lambda \geq 0.
\]

Then, according to Corollary 1 in \cite{CPP}, the descending ladder height subordinator $\widehat{H}$ is a killed Lamperti stable subordinator with characteristics $(\alpha(1-\rho), 0)$, i.e. a subordinator whose L\'evy measure is given by
\[
\Pi_{\widehat{H}}(dx)=\frac{\sin (\pi \alpha(1-\rho))}{\pi}
\frac{d x}{(e^x-1)^{1+\alpha(1-\rho)}} , \qquad x>0,
\]
and killed at unit rate.

On the other hand from  (\ref{eqvigon}), we have that the L\'evy measure of the upward ladder height subordinator $H$ satisfies
\[
\overline{\Pi}_{H}(x)=
\alpha(1-\rho)\int_{0}^{\infty} dy\,(1-e^{-y})^{\alpha(1-\rho)-1}e^{-y} \int_{x+y}^{\infty}\frac{e^{(\alpha(1-\rho)+1 )u}}{(e^u-1)^{\alpha+1}}\ du, \qquad x>0.
\]

Perfoming the above integral, we get
\[
\overline{\Pi}_{H}(x)=
\frac{\Gamma(\alpha\rho)\Gamma(\alpha(1-\rho)+1)}{\Gamma(\alpha+1)}\frac{1}{(e^{x}-1)^{\alpha\rho}},
\] 
which implies that $H$ is a  subordinator  whose L\'evy measure is given by
\[
\Pi_{H}(dx)=
\frac{\Gamma(\alpha\rho+1)\Gamma(\alpha(1-\rho)+1)}{\Gamma(\alpha+1)}\frac{e^x }{(e^x-1)^{1+\alpha\rho}}dx,
\]
that is to say a Lamperti stable subordinator with characteristics $(\alpha\rho,1)$. Since the stable process conditioned to stay positive does not creep, we deduce that  $\xi^\uparrow$ does not creep either and from Theorem VI.19 in \cite{be1} the subordinator $H$ has no drift.  Hence from Corollary 1 in \cite{CPP}, the Laplace exponent of $H$ is as follows
\[
\kappa(0,\lambda)=\frac{
\pi
}{\sin(\pi \alpha\rho)}\frac{\Gamma(\alpha(1-\rho)+1)}{\Gamma(\alpha+1)}\frac{\Gamma(\lambda+\alpha\rho)}{\Gamma(\lambda)}, \qquad \lambda\ge 0,
\]
which implies, from (\ref{doubleLT}) ,  that
\begin{equation}\label{ascrfls}
V(dx)=\frac{\sin(\pi \alpha\rho)}{
\pi
}\frac{\Gamma(\alpha+1)}{\Gamma(\alpha\rho)\Gamma(\alpha(1-\rho)+1)} (1-e^{-x})^{\alpha\rho-1}d x.
\end{equation}
It is important to note that the above discussion provides a new explicit example of the  spatial Wiener-Hopf factorization which we formally state as a proposition.

\begin{proposition}\label{WH1}
For any Lamperti-stable process $\xi^\uparrow$ with characteristics $(\alpha, \alpha(1-\rho)+1, \alpha\rho)$, its  characteristic exponent, $\Psi_{\xi^\uparrow}(\lambda):=-\log \mathbb{E}(e^{i\lambda \xi_1^\uparrow})$, enjoys the following Wiener-Hopf factorization,
\begin{eqnarray*}
\Psi_{\xi^\uparrow}(\lambda)&=&\frac{
\pi}{\sin (\pi \alpha \rho) \Gamma(\alpha+1)} \frac{\Gamma(-i\lambda +\alpha\rho)\Gamma(\alpha(1-\rho)+1+i\lambda)}{\Gamma(-i\lambda)\Gamma(i\lambda+1)}\\
&=& \frac{
\pi
}{\sin(\pi \alpha\rho)}\frac{\Gamma(\alpha(1-\rho)+1)}{\Gamma(\alpha+1)}\frac{\Gamma(-i\lambda+\alpha\rho)}{\Gamma(-i\lambda)}
\times
\frac{1}{
\Gamma(\alpha(1-\rho)+1)}\frac{\Gamma(\alpha(1-\rho)+1+i\lambda)}{\Gamma(i\lambda+1)}
\end{eqnarray*}
for $\lambda\in\mathbb{R}$ where the first equality holds up to a multiplicative constant.
\end{proposition}

Note that the above factorization also provides an alternative way of computing the characteristic exponent of such Lamperti-stable processes to the methods employed, for example, in \cite{CPP} and \cite{Patie}. This factorization should also be seen as a special case of the Wiener-Hopf factorization of the $\beta$-class of L\'evy processes appearing in the concurrent work of Kuznetsov \cite{kut}.

Now that we are in possession of the potential measures $V$ and $\widehat{V}$, we may marginalize the quintuple law at first passage times (Theorem 1) and  obtain a new identity for the Lamperti process $\xi^\uparrow$ which is given below.
\begin{corollary} For $y\in[0,x]$, $v\ge y$ and $u>0$,
\[
\begin{split}
&\hspace{-1cm}\mathbb{P}( 
\xi^\uparrow_{\tau^+_{x}}-x\in {d}u, \, 
x-\xi^\uparrow_{\tau^+_{x}-}\in {d}v,  \,
x-\overline{\xi}^\uparrow_{\tau^+_{x}-}\in
{d}y) \\
&=\frac{\sin (\pi\alpha\rho)}{\pi}\frac{\Gamma(\alpha+1)}{\Gamma(\alpha\rho)\Gamma(\alpha(1-\rho))}(1-e^{-x+y})^{\alpha\rho-1}(1-e^{-v+y})^{\alpha(1-\rho)-1}\\
&\hspace{6cm}\cdot e^{-v+y} e^{(\alpha(1-\rho)+1)(u+v)}(e^{u+v}-1)^{-\alpha-1}dy\, dv\, du.
\end{split}
\]
\end{corollary}


\noindent Similarly, from Corollary 2, we obtain a quadruple law for the last passage time of $\xi^\uparrow$. 

\begin{corollary}\label{firstpassagexiup}
For $v>0$,  $0\leq y<x+v$, $w\geq u>0$,
\[
\begin{split}
&\mathbb{P}(
- \underline{\xi}^\uparrow_{\infty} \in dv, \,
\underrightarrow{\xi^\uparrow}_{U_x} - x\in{d}u, \,
x-\xi^\uparrow_{U_x-}\in {d}y,\, %
\xi^\uparrow_{U_x}-x\in {d}w 
) \\
&=\frac{\alpha(1-\rho)\sin(\pi\alpha\rho)}{\pi}\frac{\Gamma(\alpha+1)}{\Gamma(\alpha\rho)\Gamma(\alpha(1-\rho))}(1-e^{-x-v+y})^{\alpha\rho-1}\\
&\hspace{2cm}\cdot \big((1-e^{-v})(1-e^{-w+u})\big)^{\alpha(1-\rho)-1}e^{-v-w+u}
 e^{(\alpha(1-\rho)+1)(y+w)}(e^{y+w}-1)^{-1-\alpha}dv\,dy\,dw\,du.
\end{split}
\]
\end{corollary}


Further, we may compute the triple  law at last passage times for the Lamperti-stable $\xi^\uparrow$ conditioned to stay positive starting from $0$  with the help of Theorem 2.

\begin{corollary}For $0<y\leq x$ and $0<u\le w$
\[
\begin{split}
&\hspace{-1cm}\p^{\uparrow}(\underrightarrow{\xi^\uparrow}_{U_x}-x\in du, x-\xi^\uparrow_{U_x-}\in dy, \xi^\uparrow_{U_x}-x\in d w)\\
&=\frac{\sin (\pi\alpha\rho)}{\pi}\frac{\Gamma(\alpha+1)}{\Gamma(\alpha\rho)\Gamma(\alpha(1-\rho))}(1-e^{-x+y})^{\alpha\rho-1}(1-e^{-w+u})^{\alpha(1-\rho)-1}\\
&\hspace{5cm}\cdot e^{-w+u}e^{(\alpha(1-\rho)+1)(y+w)}(e^{y+w}-1)^{-\alpha-1}du\, dy\,dw. 
\end{split}
\]  
\end{corollary}


\noindent Moreover, when the Lamperti-stable process conditioned to stay positive starts from a positive state, Corollary 1 give us the following explicit identity. 
\begin{corollary}
For $x>0$, $0< v< z\land x$, $0<y\leq x-v$ and $0<u\leq w$
\[
\begin{split}
\p^{\uparrow}_z(\underline{\xi}^\uparrow_{\infty}\in& dv, \underrightarrow{\xi^\uparrow}_{U_x}-x\in du, x-\xi^\uparrow_{U_x-}\in dy, \xi^\uparrow_{U_x}-x\in d w)\\
&=K_2(x,y)(1-e^{-x+v+y})^{\alpha\rho-1}\Big((1-e^{-z+v})(1-e^{-w+u})\Big)^{\alpha(1-\rho)-1}\\
&\hspace{4cm}  \cdot e^{(\alpha(1-\rho)+1)(y+w)}e^{-z-w+v+u}(e^{w+y}-1)^{-\alpha-1} dv \,du\, dy\,dw.
\end{split}
\]
The normalizing constant $K_2(x,z)$ (which depends on $x$ and $z$) makes the right-hand side of the previous identity a distribution and following a  quadruple intergal can be shown to be
\[
K_2(x,z)=\frac{sin (\pi\alpha\rho)}{\pi}\,\frac{\alpha (1-\rho)\Gamma(\alpha +1)}{\Gamma(\alpha\rho)\Gamma(\alpha(1-\rho))}\bigg(1-\left(0\lor \frac{1-e^{-z+x}}{1-e^{-z}}\right)^{\alpha(1-\rho)}\bigg)^{-1}.
\]  
  
\end{corollary}


Finally, we note that the process $\xi^\uparrow$ is not arithmetic and that 
\[
\mu_+=\kappa'(0,0^+)=\frac{
\pi 
}{\sin(\pi\alpha\rho)}\frac{\Gamma(\alpha\rho)\Gamma(\alpha(1-\rho)+1)}{\Gamma(\alpha+1)}<\infty.
\]


Therefore, from Theorem 3 (i), the random variable $(x-\overline{\xi}^\uparrow_{\tau^+_x-},x-\xi^\uparrow_{\tau^+_x-}, \xi^\uparrow_{\tau^+_x}-x)$ converges weakly towards a non-degenerate random variable which is given in the next corollary. 
\begin{corollary}For $0\leq u\leq v$, $w\geq 0$.
\[
\begin{split}
&\lim_{x\to\infty}\mathbb{P}\left(x-\overline{\xi}^\uparrow_{\tau^+_x-}>u,x-\xi^\uparrow_{\tau^+_x-}>v, \xi^\uparrow_{\tau^+_x}-x>w\right)\\
&=\frac{\sin (\pi\alpha\rho)}{\pi}\frac{\Gamma(\alpha+1)}{\Gamma(\alpha\rho)\Gamma(\alpha(1-\rho))}\int_{0}^{v-u}dy\int_y^\infty dz\, e^{-z}(1-e^{-z})^{\alpha(1-\rho)-1}\int_{\omega +z+v-y}^{\infty} \frac{e^{(\alpha(1-\rho)+1)l}}{(e^l-1)^{\alpha+1}} d l\\
&\hspace{5cm}+\frac{\sin(\pi\alpha\rho)}{\pi}\int_v^\infty\frac{e^{-\alpha\rho (\omega +y)}}{(1-e^{-(\omega+y)})^{\alpha\rho}} dy.
\end{split}
\]
\end{corollary}


\subsection{Conditioned stable processes and first passage times} In this example, we are interested in computing the triple law at first passage times of  stable processes conditioned to stay positive.  Note that the results in Section \ref{5-7} do not cover this eventuality. However, thanks to the Lamperti transformation, we can recover the required identities from some of the conclusions in the previous subsection. To this end we keep with our earlier notation so that  $X$ is a stable process of index $\alpha\in(0,2)$ enjoying  positive  jumps.


Taking note of the the form of the L\'evy measure of $\xi^{\uparrow}$ and the renewal functions (\ref{desrfls}) and (\ref{ascrfls}), after some algebra, we get from Corollary \ref{cor-pmasp}  the following result.
\begin{corollary}\label{13} Let $b>x>0$. For $u\in[0,b-x]$, $v\in[u,b)$ and $y>0$,
\[
\begin{split}
&\mathbb{P}_x^{\uparrow}( b-
\overline{X}_{\tau^+_{b}-}\in {d}u, \, 
b-X_{\tau^+_{b}-}\in {d}v,  \,X_{\tau^+_{b}}-b\in
{d}y)\\
& =\frac{\sin (\pi\alpha\rho)}{\pi}\frac{\Gamma(\alpha+1)}{\Gamma(\alpha\rho)\Gamma(\alpha(1-\rho))}\frac{(b-x-u)^{\alpha\rho-1}(v-u)^{\alpha(1-\rho)-1}(b-v)^{\alpha\rho}(y+b)^{\alpha(1-\rho)}}{(b-u)^{\alpha}(y+v)^{\alpha+1}}du\, dv\, dy,
\end{split}
\]

\end{corollary}


\noindent  We obtain similarly from Corollary \ref{cor-pmasp} the following formula for the stable process conditioned to stay positive starting from $0$

\begin{corollary} For $u\in [0,b]$, $v\in [0,u]$, $w>b>0$
\[
\begin{split}
&\mathbb{P}^{\uparrow}( 
\overline{X}_{\tau^+_{b}-}< u , \, 
X_{\tau^+_{b}-}< v,  \,X_{\tau^+_{b}}>w)\\ 
&=\frac{\sin (\pi\alpha\rho)}{\pi}\frac{\Gamma(\alpha+1)}{\Gamma(\alpha\rho)\Gamma(\alpha(1-\rho))}\int_0^{\log (u/v)} dy\int_y^\infty dl (1-e^{-l})^{\alpha(1-\rho)-1}e^{-l}\int^{\infty}_{\log(w/v)+l -y}\frac{e^{-\alpha\rho x}}{(1-e^{-x})^{\alpha+1}}dx\\
&\qquad + \frac{\sin(\pi\alpha\rho)}{\pi}\int_{\log(b/v)}^{\infty}\frac{b^{\alpha\rho}dy}{(e^yw-b)^{\alpha\rho}}.
\end{split}
\]
\end{corollary}


\subsection{Lamperti-stable processes: II}

Next we return to Lamperti-stable processes and make use of some of the results in the previous section to push further  more explicit identities.
To this end, recall that the law of  a stable process conditioned to stay positive at time $t>0$ when issued from $x>0$ is defined via the transformation
\begin{equation}
\mathbb{P}_x^{\uparrow}(X_{t}\in dz)=
\left(\frac{z}{x}\right)^{\alpha\rho}\mathbb{P}_x(X_t\in dz\,,t<\tau^{-}_0)
\,,\;\;\;t\ge0,\,\;\;\;z>0.
\label{doob}
\end{equation}
where  $\tau^{-}_0=\inf\{t>0:X_t\le0\}$.  It is well known that by the optional sampling theorem the latter identity extends to finite stopping times and hence
\[
\begin{split}
\mathbb{P}_x^{\uparrow}( b-
\overline{X}_{\tau^+_{b}-}&\in {d}u, \, 
b-X_{\tau^+_{b}-}\in {d}v,  \,X_{\tau^+_{b}}-b\in
{d}y)\\
&=\left(\frac{b+y}{x}\right)^{\alpha\rho}\mathbb{P}_x( b-
\overline{X}_{\tau^+_{b}-}\in {d}u, \, 
b-X_{\tau^+_{b}-}\in {d}v,  \,X_{\tau^+_{b}}-b\in
{d}y, \, \tau^+_b<\tau^-_0).
\end{split}
\]
Taking account of the identity established in Corollary \ref{13}, we deduce the following new identity  which  extends the  main result of Rogozin \cite{ro}.
\begin{corollary}For $u\in[0,b-x]$, $v\in[u,b)$ and $y>0$,
\begin{equation}\label{genro}
\begin{split}
&\mathbb{P}_x( b-
\overline{X}_{\tau^+_{b}-}\in {d}u, \, 
b-X_{\tau^+_{b}-}\in {d}v,  \,X_{\tau^+_{b}}-b\in
{d}y, \, \tau^+_b<\tau^-_0)\\
& =\frac{\sin (\pi\alpha\rho)}{\pi}\frac{\Gamma(\alpha+1)}{\Gamma(\alpha\rho)\Gamma(\alpha(1-\rho))}\frac{x^{\alpha\rho}(b-x-u)^{\alpha\rho-1}(v-u)^{\alpha(1-\rho)-1}(b-v)^{\alpha\rho}(y+b)^{\alpha(1-2\rho)}}{(b-u)^{\alpha}(y+v)^{\alpha+1}}du\, dv\, dy.
\end{split}
\end{equation}
\end{corollary}


Having established the above corollary, we may now use it to extract even more identities for Lamperti-stable processes. To begin with, we  will  follow the same line of reasoning used in Theorem 2 in \cite{CKP}  in order to get  a similar identity for the Lamperti-stable $\xi^{\uparrow}$.  To this end, we set  for $-\infty<u\leq 0<b<\infty$,
\[
T^{\uparrow+}_b=\inf\Big\{t\ge 0: \xi^{\uparrow}_t\ge b\Big\} \quad \textrm{ and }\quad T^{\uparrow-}_u=\inf\Big\{t\ge 0: \xi^{\uparrow}_t\le u\Big\}.
\]
From the Lamperti representation of $(X,\mathbb{P}^{\uparrow})$ and identity (\ref{doob}), we get for $0<\theta\le \phi<b-u$ and $\eta>0,$ 
\[
\begin{split}
\mathbb{P}&\Big( b-
\overline{\xi^{\uparrow}}_{T^{\uparrow+}_{b}-}<\theta, \, 
b-\xi^{\uparrow}_{T^{\uparrow+}_{b}-}>\phi,  \,\xi^{\uparrow}_{T^{\uparrow+}_{b}}-b<\eta, \, T^{\uparrow+}_b<T^{\uparrow-}_u\Big)\\
&=\mathbb{P}^{\uparrow}_1\Big( e^b-
\overline{X}_{\tau^+_{e^b}-}<e^b-e^{b-\theta}, \, 
e^b-X_{\tau^+_{e^b}-}>e^b-e^{b-\phi},  \,X_{\tau^+_{e^b}}-e^b<e^{\eta+b}-e^b, \, \tau^+_{e^b}<\tau^{-}_{e^u}\Big)\\
&=\int_0^{e^b-e^{b-\theta}} d x \int_{e^b-e^{b-\phi}}^{e^b-e^u} d y \int_0^{e^{\eta+b}-e^b} d z \Big(z+e^b\Big)^{\alpha\rho}\\
& \qquad\times \mathbb{P}_1\Big( e^b-
\overline{X}_{\tau^+_{e^b}-}\in dx, \, 
e^b-X_{\tau^+_{e^b}-}\in dy,  \,X_{\tau^+_{e^b}}-e^b\in dz, \, \tau^+_{e^b}<\tau^{-}_{e^u}\Big)\\
&=\int_0^{e^b-e^{b-\theta}} d x \int_{e^b-e^{b-\phi}}^{e^b-e^u} d y \int_0^{e^{\eta+b}-e^b} d z \Big(z+e^b\Big)^{\alpha\rho}\\
& \qquad\times \mathbb{P}_{1-e^u}\Big( h-
\overline{X}_{\tau^+_{h}-}\in dx, \, 
h-X_{\tau^+_{h}-}\in dy,  \,X_{\tau^+_{h}}-h\in dz, \, \tau^+_{h}<\tau^{-}_{0}\Big),\\
\end{split}
\]
where $h=e^b-e^u$.  From the identity (\ref{genro}) and some straightforward computations, we obtain the following identity for $\xi^{\uparrow}$, which generalizes Theorem 2 in \cite{CKP}.
\begin{corollary}
For $\theta \in [0,b]$, $\theta\le \phi <b-u$ and $\eta>0$
\[
\begin{split}
&\mathbb{P}\Big( b-
\overline{\xi^{\uparrow}}_{T^{\uparrow+}_{b}-}\in d\theta, \, 
b-\xi^{\uparrow}_{T^{\uparrow+}_{b}-}\in d\phi,  \,\xi^{\uparrow}_{T^{\uparrow+}_{b}}-b\in d\eta, \, T^{\uparrow+}_b<T^{\uparrow-}_u\Big)\\
&= \frac{\sin (\pi\alpha\rho)}{\pi}\frac{\Gamma(\alpha+1)}{\Gamma(\alpha\rho)\Gamma(\alpha(1-\rho))}e^b(1-e^u)^{\alpha\rho} e^{-\theta-\phi}e^{ (\alpha\rho+1)\eta}(e^{b-\theta}-1)^{\alpha\rho-1}(e^{-\theta}-e^{-\phi})^{\alpha(1-\rho)-1}\\
&\hspace{4cm}\cdot (e^{b-\phi}-e^u)^{\alpha\rho} (e^{b+\eta}-e^u)^{\alpha(1-2\rho)}(e^{b-\theta}-e^u)^{-\alpha}(e^\eta-e^{-\phi})^{-\alpha-1} d\theta\, d \phi\, d\eta.
\end{split}
\]

\end{corollary}


According to Caballero and Chaumont \cite{CC1},  stable processes when initiated from a positive position and killed at $\tau_0^-$ are also positive self-similar Markov processes. Such processes also enjoy a transformation of the kind (\ref{lamp}), but the underling L\'evy process in the transformation is killed at an independent and exponentially distributed time. In the case at hand, the underlying L\'evy process is a Lamperti-stable process with characteristics $(\alpha, 1, \alpha)$ and the killing  rate is $c_-\alpha^{-1}$. Let us denote the latter process by $\xi^*$ and set  for $-\infty<u\leq 0<b<\infty$,
\[
T^{*+}_b=\inf\Big\{t\ge 0: \xi^{*}_t\ge b\Big\} \quad \textrm{ and }\quad T^{*-}_u=\inf\Big\{t\ge 0: \xi^{*}_t\le u\Big\}.
\]

Similar arguments to those used above give us the following new identity for $\xi^{*}$, which generalize Theorem 3 in \cite{CKP}.
\begin{corollary}
For $\theta \in [0,b]$, $\theta\le \phi <b-u$ and $\eta>0$,
\[
\begin{split}
\mathbb{P}&\Big( b-
\overline{\xi^{*}}_{T^{*+}_{b}-}\in d\theta, \, 
b-\xi^{*}_{T^{*+}_{b}-}\in d\phi,  \,\xi^{*}_{T^*_{b}}-b\in d\eta, \, T^{*+}_b<T^{*-}_u\Big)\\
&= \frac{\sin (\pi\alpha\rho)}{\pi}\frac{\Gamma(\alpha+1)}{\Gamma(\alpha\rho)\Gamma(\alpha(1-\rho))}e^{b(1-\alpha\rho)}(1-e^u)^{\alpha\rho} e^{-\theta-\phi+\eta}(e^{b-\theta}-1)^{\alpha\rho-1} (e^{-\theta}-e^{-\phi})^{\alpha(1-\rho)-1}\\
&\hspace{4cm}\cdot(e^{b-\phi}-e^u)^{\alpha\rho} (e^{b+\eta}-e^u)^{\alpha(1-2\rho)}(e^{b-\theta}-e^u)^{-\alpha}(e^\eta-e^{-\phi})^{-\alpha-1} d\theta\, d \phi\, d\eta.
\end{split}
\]
\end{corollary}


Finally, we consider the stable process $X$ conditioned to hit $0$ continuously. This process is defined as a Doob $h$-transform with respect to the function $h(x)=\alpha(1-\rho)x^{\alpha(1-\rho)-1}$ which is excessive for the killed stable process at $\tau^-_0$. Moreover, the latter process is also a positive self-similar Markov process. According to Caballero and Chaumont \cite{CC1} a Lamperti transformation of the kind (\ref{lamp}) exists where the underlying L\'evy proces, denoted by $\xi^{\downarrow}$, is a 
Lamperti-stable process with characteristics $(\alpha, \alpha(1-\rho), \alpha\rho +1)$.

We set for $-\infty<u\leq 0<b<\infty$,
\[
T^{\downarrow+}_b=\inf\Big\{t\ge 0: \xi^{\downarrow}_t\ge b\Big\} \quad \textrm{ and }\quad T^{\downarrow-}_u=\inf\Big\{t\ge 0: \xi^{\downarrow}_t\le u\Big\}.
\]
The following new identity for $\xi^{\downarrow}$ follows in a similar spirit to the calculations for $\xi^{\uparrow}$ and generalizes Theorem 4 in \cite{CKP}.
\begin{corollary}
For $\theta \in [0,b]$, $\theta\le \phi <b-u$ and $\eta>0$,
\[
\begin{split}
\mathbb{P}&\Big( b-
\overline{\xi^{\downarrow}}_{T^{\downarrow+}_{b}-}\in d\theta, \, 
b-\xi^{\downarrow}_{T^{\downarrow+}_{b}-}\in d\phi,  \,\xi^{\downarrow}_{T^{\downarrow+}_{b}}-b\in d\eta, \, T^{\downarrow+}_b<T^{\downarrow-}_u\Big)\\
&= \frac{\sin (\pi\alpha\rho)}{\pi}\frac{\Gamma(\alpha+1)}{\Gamma(\alpha\rho)\Gamma(\alpha(1-\rho))}(1-e^u)^{\alpha\rho} e^{-\theta-\phi}e^{ \alpha\rho\eta}(e^{b-\theta}-1)^{\alpha\rho-1} (e^{-\theta}-e^{-\phi})^{\alpha(1-\rho)-1}\\
&\hspace{4cm}\cdot  (e^{b-\phi}-e^u)^{\alpha\rho}(e^{b+\eta}-e^u)^{\alpha(1-2\rho)}(e^{b-\theta}-e^u)^{-\alpha}(e^\eta-e^{-\phi})^{-\alpha-1} d\theta\, d \phi\, d\eta,
\end{split}
\]
\end{corollary}


\subsection{Lamperti-stable subordinators, philanthropy and hypergeometric L\'evy processes}
Here we will use the, previously unexploited, theory of philanthropy due to Vigon \cite{vigt} in order  to construct a new family of L\'evy processes, which we shall denote hypergeometric L\'evy processes,  for which we may compute triple laws at first and last passage times. 

According to Vigon's theory of philanthropy, a subordinator is called philanthropist if its L\'evy measure has a decreasing density on $\R_+$. Moreover,  given any two subordinators $H$ and $\widehat{H}$ which are philanthropists there exist a L\'evy process $X$ such that $H$ and $\widehat{H}$ have the same law as the ascending and descending ladder height processes of $X$, respectively.  Moreover, the L\'evy measure of $X$ satisfies the following identity
\begin{equation}\label{phillm}
\overline{\Pi}^+_X(x)=\int_0^{\infty}\Pi_H(x+d u)\overline{\Pi}_{\widehat{H}}(u)+\widehat{\delta}\,\pi_H(x)+\widehat{k}\,\overline{\Pi}_{H}(x), \quad x>0,
\end{equation}
where $(\widehat{k},\widehat{\delta},\Pi_{\widehat{H}})$ are the characteristics of $\widehat{H}$, $\Pi_H$ denotes the L\'evy measure of $H$ and $\pi_H$ its corresponding density.  By symmetry, an obvious analogue of (\ref{phillm}) holds for the negative tail $\overline{\Pi}_X^-(x): = \Pi_X(-\infty, x)$, $x<0$.

Recall that a Lamperti-stable subordinator with characteristics $(\varrho, \gamma)$ 
is a (possibly killed) subordinator with no drift component and  L\'evy measure  given by (\ref{LS-sub}).
From the form of the latter it is clear that Lamperti-stable subordinators are philanthropists. For simplicity, in what follows we will assume that the constant $c=1$.  
 
Let $\widehat{H}$ be a Lamperti-stable subordinator with characteristics $(\varrho, \beta)$ which is killed at rate 
\[
\frac{\Gamma(1-\varrho)}{\varrho}\frac{\Gamma(1-\beta +\varrho)}{\Gamma(1-\beta)},
\]
and $H$ a Lamperti subordinator with characteristics $(\gamma, 1)$ with no killing, where  $\varrho, \gamma\in(0,1)$ and $\beta\le 1$.  Let us denote by $X$ the L\'evy process whose ascending and descending ladder height processes have  the same law as $H$ and $\widehat{H}$, respectively. From (\ref{phillm}), the L\'evy measure of $X$ is such that
\[
\overline{\Pi}^+_X(x)=\int_x^{\infty}\frac{e^u}{(e^u-1)^{\gamma+1}}\int_{u-x}^{\infty}\frac{e^{\beta z}}{(e^z-1)^{\varrho+1}}dz\, du- \frac{\Gamma(1-\varrho)}{\gamma\varrho}\frac{\Gamma(1-\beta +\varrho)}{\Gamma(1-\beta)} (e^u-1)^{-\gamma}.
\]


Applying the binomial expansion twice, we obtain that  
\[
\begin{split}
&\int_x^{\infty}\frac{e^u}{(e^u-1)^{\gamma+1}}\int_{u-x}^{\infty}\frac{e^{\beta z}}{(e^z-1)^{\varrho+1}}dz\, du\\
&\hspace{1cm}=\frac{1}{(\varrho+1-\beta)}\sum_{n,k=0}^{\infty}\frac{(\varrho+1)_n(\varrho+1-\beta)_n}{n!(\varrho+2-\beta)_n}\frac{(\gamma+1)_k(\gamma+1+\varrho+n-\beta)_k\, e^{-x(\gamma+k)}}{(\gamma+1+\varrho+n-\beta)k!(\gamma+2+\varrho+n-\beta)_k},
\end{split}
\]
where $(z)_n=\Gamma(z+n)/\Gamma(z)$, $z\in \mathbb{C}$. Putting the pieces together, we may write the L\'evy measure of $X$ as follows
\[
\begin{split}
\overline{\Pi}^+_X(x)&=\frac{1}{(\varrho+1-\beta)}\sum_{n,k=0}^{\infty}\frac{(\varrho+1)_n(\varrho+1-\beta)_n}{n!(\varrho+2-\beta)_n}\frac{(\gamma+1)_k(\gamma+1+\varrho+n-\beta)_k\,e^{-x(\gamma+k)}}{(\gamma+1+\varrho+n-\beta)k!(\gamma+2+\varrho+n-\beta)_k}\\
&\qquad -\frac{\Gamma(1-\varrho)}{\gamma\varrho}\frac{\Gamma(1-\beta +\varrho)}{\Gamma(1-\beta)} (e^u-1)^{-\gamma}.
\end{split}
\]


For simplicity, we denote by $f$ for the  density of L\'evy measure $\Pi_X$ on $\R_+$ which can be obtained by differentiating  the above expression.  

It is important to  note that the process $X$  has no gaussian component  and that $\Pi_X$ also satisfies
\[
\overline{\Pi}^-_X(x)=\sum_{n,k=0}^{\infty}\frac{(\gamma)_n}{\gamma n!}\frac{(\varrho+1)_k(\gamma+\varrho+1-\beta+n)_k}{k!(\gamma+2+\varrho-\beta+n)_k(\gamma+\varrho+1-\beta+n)}e^{-x(\varrho+1+k-\beta)},\qquad \textrm{for } x<0.
\]
Moreover the process $X$ drift to $\infty$, when $\beta< 1$, and oscillates, when $\beta=1$. In the latter case, the form  of the L\'evy measure of $X$ is much simpler and is given by 
\[
\overline{\Pi}^+_X(x)=\frac{\Gamma(\gamma+\varrho)\Gamma(1-\varrho)}{\varrho\Gamma(\gamma+1)}\frac{1}{(e^x-1)^{\gamma}(1-e^{-x})^{\varrho}}, \qquad x>0.
\] 

We call the process $X$ a  hypergeometric L\'evy process with characteristics $(\varrho, \gamma, \beta)$.  When the characteristics of the hypergeometric process are such that $\varrho=\alpha(1-\rho), \gamma=\alpha\rho$ and $\beta=0$, the process $X$ is the Lamperti-stable process with characteristics $(\alpha, \alpha(1-\rho)+1, \alpha\rho)$ studied in Subsection \ref{lamperti-I}.

From Corollary 1 in \cite{CPP}, we know that the Laplace exponent of $\widehat{H}$ satisfies
\[
\hat{\kappa}(0,\lambda)= \frac{\Gamma(1-\varrho)}{\varrho}\frac{\Gamma(\lambda+1-\beta+\varrho)}{\Gamma(\lambda+1-\beta)} \qquad \lambda\ge 0, 
\]     
and from (\ref{doubleLT}), we deduce that
\[
\widehat{V}(d x)= \frac{\varrho \sin(\pi\varrho)}{\pi}e^{(\beta-1)x}(1-e^{-x})^{\varrho-1} dx.
\]
Similarly  for the subordinator $H$, we have
\[
\kappa(0,\lambda)= \frac{\Gamma(1-\gamma)}{\gamma}\frac{\Gamma(\lambda+\gamma)}{\Gamma(\lambda)} \qquad \lambda\ge 0, 
\]     
and 
\[
V(d x)= \frac{\gamma \sin(\pi\gamma)}{\pi}(1-e^{-x})^{\gamma-1} dx.
\]
Hence we identify the following Wiener-Hopf factorization which generalizes Proposition \ref{WH1}.
\begin{proposition}
For any hypergeometric  L\'evy process $X$ with characteristics $(\varrho, \gamma, \beta)$, its  characteristic exponent, $\Psi_{X}(\lambda)= - \log\mathbb{E}(e^{i\lambda X_1})$, enjoys the following Wiener-Hopf factorization,
\begin{eqnarray*}
\Psi_X(\lambda)&=&\frac{\Gamma(1-\gamma)\Gamma(1-\varrho)}{\varrho\gamma}\frac{\Gamma(-i\lambda+\gamma)\Gamma(i\lambda +1-\beta+\varrho)}{\Gamma(-i\lambda)\Gamma(i\lambda+1-\beta)}\\
&=&\frac{\Gamma(1-\gamma)}{\gamma}\frac{\Gamma(-i\lambda+\gamma)}{\Gamma(-i\lambda)} 
\times
\frac{\Gamma(1-\varrho)}{\varrho}\frac{\Gamma(i\lambda+1-\beta+\varrho)}{\Gamma(i\lambda+1-\beta)}
\end{eqnarray*}
where the first  equality hold up to a multiplicative constant.
\end{proposition}
Marginalizing the quintuple law at first passage times (Theorem \ref{quintupleDK}), we obtain one of but many identities for the hypergeometric L\'evy process $X$.  
\begin{corollary}
For $y\in[0,x]$, $v\ge y$ and $u>0$, 
\[
\begin{split}
&\mathbb{P}( 
X_{\tau^+_{x}}-x\in {d}u, \, 
 x-X_{\tau^+_{x}-}\in {d}v,  \,
x-\overline{X}_{\tau^+_{x}-}\in
{d}y) \\
&=\varrho\gamma\frac{\sin(\pi\varrho)\sin(\pi\gamma)}{\pi^2}(1-e^{-x+y})^{\gamma-1}(1-e^{-v+y})^{\rho-1}e^{(\beta-1)(v-y)} f(u+v)\,dy\, dv\, du,
\end{split}
\]

\end{corollary}


\noindent We leave the reader to amuse him/herself with the plethora of other similar examples which can be obtained from earlier results in this paper.

\section*{Acknowledgment}
We gratefully acknowledge the financial support of EPSRC grant nr. EP/D045460/1 and Royal Society Grant number RE-MA1004. We would also like to thank two anonymous referees whose detailed commentary on an earlier version of this paper enabled us to greatly improve its readability.

\bibliographystyle{abbrv}
\bibliography{bibli}
\end{document}